\documentclass[11pt]{amsart}
\usepackage{geometry}                
\usepackage{graphicx}
\usepackage[all,cmtip,line]{xy}
\usepackage{amssymb}
\usepackage{amsmath}
\usepackage{amsthm}
\usepackage{epstopdf}
\newtheorem{defin}{Definition}
\newtheorem{prop}{Proposition}
\newtheorem{thm}{Theorem}
\newtheorem{lem}{Lemma}
\newtheorem*{thme}{Theorem [Emerton]}
\newtheorem*{thma}{Theorem A}
\newtheorem*{thmlg}{Theorem (Local to Global Compatibility)}
\newtheorem*{conjll}{Conjecture (Level Lowering)}
\newtheorem*{conj}{Conjecture}
\newtheorem*{thmll}{Theorem (Level Lowering)}
\newtheorem*{conjlr}{Conjecture (Level Raising)}
\newtheorem*{thmlr}{Theorem (Level Raising)}
\newtheorem{cor}{Corollary}
\newcommand{\wf}{W(\mathbb{F})}
\newcommand{\f}{\mathbb{F}}
\newcommand{\arwf}{\mathfrak{AR}_{W(\mathbb{F})}}
\newcommand{\vf}{V_{\f}}
\newcommand{\rbf}{R^{\Box}_{\vf}}
\newcommand{\rf}{R_{\vf}}
\newcommand{\qp}{\mathbb{Q}_p}
\newcommand{\ra}{\mathcal{R}_A}

\newcommand{\are}{\mathfrak{AR}_{E}}
\newcommand{\xf}{\mathfrak{X}^{\Box}_{\vf}}
\newcommand{\x}{\mathfrak{X}_{\vf}}
\newcommand{\xflb}{\mathfrak{X}^{\Box}_{V^l_{\mathbb{F}}}}
\newcommand{\xflbsp}{\mathfrak{X}^{\Box, sp}_{V^l_{\mathbb{F}}}}
\newcommand{\xflbn}{\mathfrak{X}^{\Box, N=0}_{V^l_{\mathbb{F}}}}

\newcommand{\xfb}{\mathfrak{X}^{\Box}_{\vf}}

\newcommand{\rbe}{R^{\Box}_{V_E}}

\newcommand{\rbfn}{R^{\Box, N=0}_{\vf}}
\newcommand{\xfn}{\mathfrak{X}^{\Box, N=0}_{\vf}}
\newcommand{\xff}{\mathfrak{X}_{\vf}}

\newcommand{\fv}{F_{\textit{v}}}
\newcommand{\gv}{G_{\fv}}
\newcommand{\rbv}{R^{\Box}_{\textit{v}}}
\newcommand{\rbs}{R^{\Box}_{\Sigma}}
\newcommand{\rbfs}{R^{\Box}_{F,S}}
\newcommand{\mbs}{\mathfrak{m}^{\Box}_{\Sigma}}
\newcommand{\mbfs}{\mathfrak{m}^{\Box}_{F, S}}
\newcommand{\mbss}{\mathfrak{m}^{\Box, 2}_{\Sigma}}
\newcommand{\mbfss}{\mathfrak{m}^{\Box, 2}_{F, S}}
\newcommand{\gqs}{G_{\mathbb{Q}, S}}
\newcommand{\gqp}{G_{\mathbb{Q}_p}}
\newcommand{\gql}{G_{\mathbb{Q}_l}}
\newcommand{\bc}{B_{cris}}

\newcommand{\q}{\mathbb{Q}}
\newcommand{\cp}{\mathbb{C}_p}
\newcommand{\ql}{\mathbb{Q}_l}
\newcommand{\vi}{\textit{v}}
\begin{document}

\title{Failure of the Local to Global Principle on the Eigencurve}
\author{Alexander G.M. Paulin}
\begin{abstract}
If $\pi_f$ is a cuspidal automorphic representation of $GL_{2/\q}$ associated to a modular form $f$, the local and global Langlands correspondences are compatible at all finite places of $\q$.  On the $p$-adic Coleman-Mazur eigencurve this principle can fail (away from $p$) under one of two conditions: on a generically principal series component where monodromy vanishes; or on a generically special component where the ratio of the Satake parameters degenerates.  We prove, under mild restrictive hypotheses,  that such points are the intersection of  generically principal series and special components. This is a geometric analogue of Ribet's level raising and lowering theorems.
\end{abstract}

\maketitle

\section{Introduction}
\noindent
Let $\pi_f$ be a cuspidal automorphic representation of $GL_{2/\q}$ associated to a cuspidal modular form $f$.  Let  $p$ be a prime number and fix $\bar{\q}$, an algebraic closure of $\q$ .  To $\pi_f$ we may naturally associate $\rho_f$, a  2-dimensional, $p$-adic potentially semi-stable representation of $G_\q:=Gal(\bar{\q}/\q)$, naturally fitting into a compatible system.   Each of these global objects naturally decomposes into local data.  

On the automorphic side, the restriction of  $\pi_f$ to the finite adeles is the restricted tensor product
$$\pi_f =\otimes' \pi_{f,l}$$
where $\pi_{f,l}$ is a smooth admissible irreducible representation of $GL_2(\q_l)$, for $l$ prime.  On the Galois side, for any prime $l$, $\rho_f$ gives rise to a 2-dimensional Frobenius semisimple Weil-Deligne representation $(\rho_{f,l}, N_f)$. 

If $\pi$ is the Tate normalised local Langlands correspondence then by work of Carayol (\cite{CAR}), Saito (\cite{TS}), et al,  local to global compatibility holds, i.e.
$$\pi(\rho_{f,l}, N_f) \cong \pi_{f, l}.$$  
In \cite{PAU} we extend this result to families of finite slope overconvergent modular forms.  More precisely,  for $N$ a natural number coprime to $p$,  let $\mathcal{E}$ be the cuspidal eigencurve (for the full Hecke algebra) of tame level $N$.  Let $l$ be a prime not equal to $p$.  By \cite{PAU}, to any  $x \in \mathcal{E}$ we may associate $\pi_{x,l}$, a smooth admissible representation of $GL_2(\q_l)$ and $(\rho_{x,l}, N_{x})$ a 2-dimensional Frobenius semisimple Weil-Deligne representation.  The central result of \cite{PAU} is:

\begin{thmlg}Away from a discrete subset $\mathcal{X}\subset \mathcal{E}$, local to global compatibility holds, i.e. for $x\in \mathcal{E} \setminus \mathcal{X}$ we have
$$\pi(\rho_{x,l}, N_x) \cong \pi_{x, l}.$$  
\end{thmlg}
In \cite{PAU} it is also shown that on the irreducible locus of $\mathcal{E}$, failure can only occur when $\pi(\rho_{x,l}, N_x)$ is one dimensional, something with cannot happen classically.  In this paper we address the geometric consequences of this failure.

Let $l$ and $p$ be two distinct rational primes. This will be the convention adopted throughout the rest of the paper. Let $S$ be a finite set of places of $\mathbb{Q}$ containing $p$, $l$ and $\infty$.   Let $\f$ be a finite field of characteristic $p$ and $\vf$ be a two dimensional $\f$-vector space equipped with a continuous, odd action of $G_{\mathbb{Q}, S}$.  Let us, furthermore, assume that $end_{\f[G_{\mathbb{Q}, S}]}\vf = \f$ and that $\vf$ is modular, associated to a cuspidal eigenform.  If $\vf$ is absolutely irreducible this last condition is automatically satisfied by Serre's conjecture (now a theorem of Khare, Wintenberger and Kisin). Let $\x$ denote the universal deformation space associated to $\vf$.

As explained in \S5, there is a closed subspace $\mathcal{E}\subset \x \times \mathbb{G}_m$, the cuspidal eigencurve lying over $\vf$ (see \S 6 of \cite{Kis4} for a more detailed Galois theoretic construction, or \S2 of \cite{KB} for a Hecke theoretic construction).   As above, to any $x\in \mathcal{E}$ we may naturally associate a 2-dimensional Frobenius semisimple Weil-Deligne representation $(\rho_{x,l}, N_{x})$.  We remark that the Hecke construction of $\mathcal{E}$ given in \cite{KB} involves glueing restricted Hecke Algebras, i.e. those away from the level.  Hence, strictly speaking, $\cp$-valued points on $\mathcal{E}$ correspond to systems of eigenvalues associated to a finite slope overconvergent cuspidal eigenforms, and not necessarily to eigenforms themselves. Hence, this \textit{eigencurve} is subtly different from the one introduced above.  There is, of course, a natural morphism between the two as explained in \S5. We make the following conjecture:

\begin{conj} Let $x\in \mathcal{E}$ be a point such that $\pi(\rho_{x,l}, N_{x})$ is one dimensional.  Then there exist irreducible components $\mathcal{Z}, \mathcal{Z}'\subset \mathcal{E}$,  generically special and principal series respectively, such that $x\in \mathcal{Z}\cap\mathcal{Z}'$.
\end{conj}
This is equivalent to a combination of the following geometric analogues of Ribet's level lowering and raising theorems:
 
 \begin{conjll}
Let $\mathcal{Z}\subset \mathcal{E}$ be a generically special (at $l$) component such that there exists $x\in \mathcal{Z}$ such that $N_{x} = 0$, then there exists $ \mathcal{Z}'\subset \mathcal{E}$, a generically principal series irreducible component, such that $x \in\mathcal{Z}'$. 
\end{conjll}
\begin{conjlr}
Let $\mathcal{Z}\subset \mathcal{E}$ be a generically principal series (at $l$) component such that there exists $x\in \mathcal{Z}$ such that $\pi(\rho_{x,l}, N_{x})$ is one dimensional.   Then there exists $ \mathcal{Z}'\subset \mathcal{E}$, a generically special irreducible component, such that $x \in\mathcal{Z}'$. 
\end{conjlr}
\noindent
We prove these conjectures under mild technical restrictions on $\vf$ and $x$.  
\\ \\
Let $V_x$ be the $2$-dimensional $p$-adic representation associated to $x$.  Let $V_x^p$ denote its restriction to a decomposition group at $p$.   We say that $x\in \mathcal{E}$ satisfies $(\star)$ if
\begin{enumerate}
\item $V_x$ and $V_x^p$ are absolutely irreducible.
\item The generalised Hodge-Tate weights of $V_x^p$ are $0$ and $k \notin -\mathbb{N}\cup\{0\}$
\end{enumerate}
Our main result is the following:
\begin{thma} 
Let $S$ be a finite set of places of $\q$ such that $\{l,p, \infty\}\subset S$ and $S \setminus \{l,p, \infty\} \neq \emptyset $.  Let $\vf$ be a 2-dimensional, mod $p$,  modular representation of $\gqs$ such that $end_{\f[G_{\mathbb{Q}, S}]}\vf = \f$.  Furthermore assume that $V_{\f}^p$ is $p$-distinguished and not a twist of an extension of the cyclotomic character by the trivial character.   Let $\mathcal{E}$ denote the  cuspidal eigencurve lying over $\vf$.  Let $x \in \mathcal{E}$ such that 
\begin{enumerate}
\item $V_x$  satisfies $(\star)$.
\item $\pi(\rho_{x,l}, N_{x})$ is one dimensional.

\end{enumerate}
Then there exist irreducible components $\mathcal{Z}, \mathcal{Z}'\subset \mathcal{E}$ generically special and principal series respectively such that $x\in \mathcal{Z}\cap\mathcal{Z}'$.
\end{thma}

We remark that $S$ must contain an auxillary place purely for technical reasons (see proposition 14). 

Let us outline the basic strategy of the proof.   In \cite{Kis4} it is proven that for any $x \in \mathcal{E}$, $V_x^p$ satisfies the property of being trianguline (see \S 3.4).  Using this Kisin (\cite{Kis4}) defines the closed rigid analytic subspace space $X_{fs} \subset \x\times \mathbb{G}_m$  using purely Galois theoretic techniques.  One should think of $X_{fs}$ as a Galois theoretic avatar of $\mathcal{E}$.  In particular there is a canonical inclusion $\mathcal{E}\subset X_{fs}$.  In \cite{Kis4} it was conjectured that this inclusion is actually an equality, giving a purely Galois theoretic construction of $\mathcal{E}$.  By recent work of Emerton (\cite{EM}) this has been shown to be true after imposing the above restrictions on  $\vf$.  Philosophically speaking this is an $R=T$ theorem.  

This allows us to study the local geometry of $\mathcal{E}$ using $X_{fs}$.  The benefit of such an approach is that the local geometry of $X_{fs}$ can be understood using Galois deformation theory (Theorem 6 of \S 4.2).  In particular we may apply  the theory of trianguline deformations as developed in \cite{BC} to determine the local geometry of $X_{fs}$. Our strategy is to first prove the level raising and lowering conjectures for $X_{fs}$ using Galois deformation theory, then using Emerton's result, deduce them for $\mathcal{E}$.  

We  remark that this strategy is reminiscent of that employed by Gee in \cite{Gee}, where level raising ard lowering results are proven using deformation theory and then transported to the realm of automorphic representations using an $R=T$ theorem.  
We should also remark that certain cases of the level raising conjecture have already been established by work of Newton (\cite{New}), in the unramified case by directly generalising the methods of Diamond and Taylor in the classical case.  Our approach is fundamentally different.

The techniques developed in this paper are well suited to studying the local geometry of Galois theoretic eigenvarieties in the higher rank case, as being developed by Pottharst.    As $R=T$ theorems improve it seems reasonable to hope that these techniques will provide a means of understanding the local geometry of higher rank automorphic eigenvarieties.  

We also remark that the techniques developed in this paper may be used to prove similar geometric level raising and lowering results for the full space of $p$-adic modular forms. More precisely, one may prove a natural analogue of theorem $A$ for the rigid analytic space associated to the \text{big} Hecke algebra acting on the the full space of cuspidal $p$-adic modular forms, under mild restrictions on $\vf$.  The strategy in this case would be the same, first proving it for the universal deformation space associated to $\vf$ and then transferring it the setting of $p$-adic modular forms using results of Boeckle (\cite{BOC}).  In a sense, this situation is easier because one does not need to impose any condition on the local deformations at $p$.

\subsection*{Arrangement of the Paper}  In \S 2 we review the basic theory of deformation of representations of profinite groups.  In \S 3.1 we study local deformations away from $p$, refining results of Gee (\cite{Gee}) and Kisin (\cite{Kis5}). In \S 3.2 we review the theory of trianguline deformation theory developed in \cite{BC}. In particular we show that under favourable circumstances Kisin's $h$-deformation functor (see \S 8 of \cite{Kis4}) and the trianguline defomation functor are naturally isomorphic.  This allows us to \textit{transport} results from \cite{BC} to the setting of \cite{Kis4}.  In \S 4.1 we prove a characteristic zero analogue of Kisin's results  on presenting global deformation rings over local deformations in positive characteristic (\cite{Kis2}).  Finally in \S4.2 we bring these ideas together to prove the level lowering and raising conjectures for $X_{fs}$.  In \S5 we review the basic theory of the eigencurve and prove level raising and lowering in this context invoking Emerton's $R=T$ result.

\subsection*{Acknowledgments}
I would like to thank Kevin Buzzard and Mathew Emerton for the numerous and helpful conversations I have had with them.

\section{Deformations of Profinite Groups}

In this section we review the basic concepts of deformations of finite dimensional representations of profinite groups.  We also recall some of the techniques needed to study the local geometry of the universal deformation space of a finite characteristic representation in the spirit of \cite{Kis5} .
\\ \\
Let $p$ be a rational prime and $\mathbb{F}$ a finite field of characteristic $p$.  Let $G$ be a profinite group and $V_{\mathbb{F}}$ be a finite dimensional $\mathbb{F}$-vector  space equipped with a continuous action of $G$.  We write $dim_{\f}(\vf) =d$.  In practice $G$ will always be a Galois group of a number field or a local field.  

Let $\wf$ be the Witt vectors over $\f$ and $\mathfrak{AR}_{W(\mathbb{F})}$ be the category of finite local, Artinian $\wf$-algebras with residue field $\f$.  For $A \in ob(\arwf)$  we define a deformation of $\vf$ to $A$ to be a finite free $A$-module, $V_A$, equipped with a continuous $A$-linear action of $G$ and an isomorphism $V_A\otimes_A \f \cong \vf$.  An isomorphism between two such deformations is a $A[G]$-module isomorphism where the induced automorphism of $\vf$ is the identity.

Using this we define the functor $\mathcal{D}_{\vf}$, which assigns to any $A\in ob(\arwf)$ the set of isomorphism classes of deformations of $\vf$ to $A$.  We remark that this functor is naturally isomorphic to the usual deformation functor defined  in terms of strict equivalence classes of liftings of $\vf$.  
\\ \\
Let $\beta_{\f}$ be a basis for $\vf$.  If $A\in ob(\arwf)$ and $V_A$ is a deformation of $\vf$, then a \textit{framing} of $V_A$ with respect to $\beta_{\f}$ is a choice of $A$-basis lifting $\beta_{\f}$. We will be interested in deformations with multiple framings.  Let $\Sigma$ be a finite index set such that  for every $\vi\in \Sigma$ we have a fixed basis $\beta_{\vi}\subset \vf$.     We define the functor $\mathcal{D}^{\Box}_{\vf}$, which assigns to every $A\in ob(\arwf)$ the set of isomorphism classes of deformations of $\vf$ to $A$ together with a \textit{framing} of $\beta_{\vi}$ for each $\vi\in \Sigma$.  There is a forgetful morphism of functors $\mathcal{D}^{\Box}_{\vf} \rightarrow\mathcal{D}_{\vf}$.  It is clear that this morphism is formally smooth by construction.
\\ \\
\noindent
From now on let us make the technical restriction that  $Hom(G, \f_p)$ is a finite dimensional $\f_p$-vector space.  We then have the crucial result originally due to Mazur (\cite{MA}) and extended by Kisin in \cite{Kis5}:

\begin{prop}
\begin{enumerate}
\item $\mathcal{D}^{\Box}_{\vf}$ is pro-representable by a complete local Noetherian $\wf$-algebra $\rbf$ called the universal framed deformation ring.
\item  If $End_{\f[G]}\vf = \f$ then $\mathcal{D}_{\vf}$ is pro-represented by a complete local $\wf$-algebra $\rf$ called the universal deformation ring.
\end{enumerate}
\end{prop}
The primary motivation for introducing the concept of a \textit{framed} deformation is to ensure representability under \textit{any} circumstances.
\\ \\
\noindent
Let $\f[\epsilon] = \f[X]/X^2$ and $ad\vf$ denote the dual numbers and the adjoint representation respectively.  The following well known lemma (\cite{MA}, \cite{Kis5}) expresses the relationship between the tangent spaces of these two functors.

\begin{lem}
\begin{enumerate}
\item There is a canonical isomorphism of finite dimensional $\f$-vector spaces:
$$\mathcal{D}_{\vf}(\f[\epsilon]) \cong H^1(G, ad\vf).$$
\item $\mathcal{D}^{\Box}_{\vf}(\f[\epsilon])$ is a finite dimensional $\f$-vector space satisfying:
$$dim_{\f}\mathcal{D}^{\Box}_{\vf}(\f[\epsilon]) = dim_{\f}\mathcal{D}_{\vf}(\f[\epsilon] +|\Sigma|d^2 - dim_{\f}(ad\vf)^G.$$
\item $\mathcal{D}_{\vf}$ (and thus $\mathcal{D}^{\Box}_{\vf}$) is formally smooth over $\wf$ if and only if $H^2(G, ad\vf)=0$.
\end{enumerate}

\end{lem}

\noindent
Frequently  the theory of deformations of finite characteristic representation of profinite groups is set up over some auxiliary coefficient-ring $\mathcal{O}$.  Recall that this means that $\mathcal{O}$ is  a complete local noetherian algebra with residue field $\f$.  Note that such a ring is canonically a $\wf$-alebgra. All of the above theory can be set up in this more restricted setting and all above results hold.  In particular the universal framed deformation in the case of $\mathcal{O}$-algebra framed deformations is given by $\rbf \hat{\otimes}_{\wf} \mathcal{O}$.  In this sense, our setup is the most general.
\\ \\ 
Let $E/\qp$ be a finite extension and $\are$ be the category of local, Artinian $E$-alegbras with residue field $E$.  Similarly let $V_E$ be a finite dimensional $E$-vector space equipped with a continuous $E$-linear action of $G$.  We define the deformation functor,  $\mathcal{D}_{V_E}$, on $\are$ in exactly the same way as above.  Similarly if $\Sigma$ is a finite index set and $\beta_{\vi}\subset V_E$ is a basis for each $\vi \in \Sigma$, we define the framed deformation functor $\mathcal{D}^{\Box}_{V_E}$ in exactly the same way as above.    All of the above results carry over to this setting replacing $\f$ by $E$.  
\\ \\
There is an elegant geometric interpretation linking deformations from characteristic $p$ and characteristic zero originally developed in \S2.3 of \cite{Kis5} and \S 9 of \cite{Kis4}.

Keeping the above notation, let $\mathfrak{X}^{\Box}_{\vf}$ be the generic fibre of the formal scheme $Spf(\rbf)$ as explained in \cite{DJ}.  $\mathfrak{X}^{\Box}_{\vf}$ is a separated rigid space over $\wf[1/p]$. By lemma 7.1.9 of \cite{DJ} the points of $\xf$ are in natural bijection with the maximal spectra of $\rbf[1/p]$.
Furthermore, if $x \in \xf$ corresponds to the maximal ideal $m \subset \rbf[1/p]$ then there is a canonical isomorphism of complete local noetherian rings 
$$\widehat{\mathcal{O}}_{\xf, x} \cong \widehat {  \rbf[1/p]_m}.$$

In particular they must have the same residue field, which is automatically a finite extension $E/ \qp$.  The canonical  morphism  $\rbf \rightarrow \widehat {  \rbf[1/p]_m}$ composed with reduction modulo the maximal ideal induces a continuous  representation of $G$ on a finite dimensional $E$-vector space $V_E$.  Furthermore, $V_E$ comes equipped with a choice of basis for each $\vi\in \Sigma$.    We can now apply the above deformation theoretic techniques to $V_E$ to give the universal framed deformation ring $\rbe$.  By lemma 2.3.2, lemma 2.3.3 and proposition 2.3.5 of \cite{Kis5} we know that 
 
 $$\rbe \cong \widehat {  \rbf[1/p]_m}.$$
Hence we may study the infinitesimal geomtry of $\xf$ by using deformation theoretic techniques. 
\\ \\
\noindent 
If $End_{\f[G]}\vf = \f$ then $\mathcal{D}_{\vf}$ is pro-representable.  Hence  there exists an associated universal deformation space $\mathfrak{X}_{\vf}$. All of the above observations remain true in this situation replacing all framed deformation functors with their unframed counterparts.  

\section{Local Deformations}

\subsection{Local Deformations away from $p$}
Our eventual aim is to study deformations of global Galois groups on $p$-adic vector spaces with prescribed behaviour at some prime $l$ not equal to $p$.  With this in mind we begin by studying  the structure of local deformation spaces associated to representations of $l$-adic Galois groups acting on finite dimensional vector spaces over a finite field $\mathbb{F}$  of characteristic $p$ where $l \neq p$.  Such deformation spaces were extensively studied in \S 2 of \cite{Gee}, adapting the methods developed in \S 3 of \cite{Kis1}. 
\subsubsection{Local Level Lowering Deformations}
We wish to study the part of the local deformation space on which monodromy vanishes. This is the geometric analogue of studying the potentially crystalline locus of the universal deformation space when $l=p$. This is a natural geometric analogue of classical level lowering, albeit more refined as it is more general than being unramified.  
\\ \\
\noindent
Let $K/\mathbb{Q}_l$ be a finite extension with residue field of cardinality $l^m$.  Fix an algebraic closure $\bar{K}$ and let  $G_K$ be the absolute Galois group, together with its natural profinite topology.  By local class field theory we know that $G_K$ satisfies the $p$-finiteness condition of \S2.
Let $V_{\mathbb{F}}$ be a $d$-dimensional $\f$-vector space equipped with a basis $\beta_{\mathbb{F}}$  and a continuous $\f$-linear action of $G_K$.  Let $\rbf$ be the complete local Noetherian $\wf$-algebra pro-representing $\mathcal{D}^{\Box}_{\vf}$ (with one framing) . Let $\xf$ denote the universal deformation space. 
\\ \\
\noindent
To $x\in \xf$ with residue field $E/\qp$ we may naturally associate a $d$-dimensional $E$-vector space $V_x$, equipped with a fixed basis and a continuous $E$-linear action of $G_K$.   

Let $I_K \subset W_K$ be the inertia subgroup and Weil group respectively.  Fix $\Phi\in W_K$, a lift of (arithmetic) Frobenius.  Let $||.||:W_K \rightarrow \mathbb{Q}^*$ be the unramified character which sends $\Phi$ to $l^m$.

A well known theorem of Grothendieck (\cite{ST}) states that there is a finite extension of $L/K$ such that $V_x$ restricted to ${I_L}$ is unipotent.  This allows us to associate to $V_x$, a $d$-dimensional Weil-Deligne representation of $W_K$ over $E$.  Such an object is a triple $(\Delta, \rho_0, N)$ where $\Delta$ is a $d$-dimensional $E$-vector space; $\rho_0 : W_K \rightarrow Aut_E(\Delta)$ is a representation whose kernel contains an open subgroup of $I_K$ and $N \in Aut_E(\Delta)$ is a nilpotent endomorphism satisfying 
$$\rho_0(\sigma)N = ||\sigma||N\rho_0(\sigma) \;\; \forall \sigma\in W_K.$$
 $N$ is  the monodromy operator associated to $V_x$ .  It is a measure of how much the initial representation of $G_K$  fails to be continuous for the discrete topology on $V_x$. In particular, if $N=0$ then $(\Delta, \rho_0)$ is isomorphic to $V_x$ as a $W_K$ representation. 
\noindent
There is a natural concept of a  Weil-Deligne representation over an arbitrary $\mathbb{Q}_p$-algebra.  There is also a natural extension of Grothendieck's result in the following sense:  
\\ \\
\noindent
As above let $\f$ be a finite field of characteristic $p$.   Let us denote by $\widehat{\arwf}$, the category of complete local Noetherian $\wf$-algebras with residue field $\f$. Let $A_o \in ob( \widehat{\arwf})  $ and $V_{A_o}$ be a finite free $A_o$-module of rank $d$ together with an  $A_o$-basis $\beta_{A_o}$ and a continuous $A_o$-linear action of $G_K$.  Write $A=A_o[1/p]$.  $V_A := V_{A_o} \otimes A$ naturally comes equipped with a continuous $A$-linear action of $G_K$ and a canonical $A$-basis  $\beta_A$, lifting $\beta_{A_o}$.  If we denote the generic fibre of $Spf(A_{o})$ by $\mathfrak{X}$ then as above,  $Specmax(A)$ is in natural bijection with $\mathfrak{X}$.   We have the following mild generalization of Grothendieck's result:

\begin{prop}
To $V_A$ we may naturally associate a Weil-Deligne representation over $A$ with the property that specializing to $x \in Specmax(A)$ recovers the classical construction.
\end{prop}
\begin{proof}
By proposition 19 of \cite{PAU} we know the result is true for affinoid alegbras.  We may admissibly cover $\mathfrak{X}$ by affinoids. Glueing each of these Weil-Deligne representations gives the desired Weil-Deligne representation over $A$.
\end{proof}
We write $N_{V_A}$ for the monodromy  operator of this Weil-Deligne representation over $A$.  Because we have a fixed $A$-basis $\beta_A$ we have the canonical inclusion $N_{V_A} \in M_d(A)$.  Furthermore,  it is true by construction that,  $N_{V_A} \in M_d(A_o)$.  Let $I \subset A^o$ denote the ideal generated by the entries of $N_{V_A}$.  We write $A_o^{N=0} = A_o/I$ and let $V_{A_o^{N=0}}  = V_{A_o} \otimes A_o^{N=0}$ together with its natural  $G_K$-action and canonical $A_o^{N=0}$-basis.   We should observe that it is perfectly possible that $A_o^{N=0}$ is trivial.   By construction the induced action of ${I_K}$ on $V_{A_o^{N=0}}$ factors through a finite quotient.  
\\ \\
\noindent
Now set $A_o = \rbf$. The above construction gives a closed subspace $\xfn \subset \xf$.  $\xfn$ is the generic fibre of $Spf(\rbfn)$ and by construction, $x \in \xfn $ if and only if the monodromy operator associated to $V_x$ is trivial.   

Observe that the Weil-Deligne representation carried by $\xfn$ must factor through a finite quotient of $I_K$.  Hence there exists a finite extension $L/K$ such that this Weil-Deligne representation factors through the finite inertia group $I_{L/K}$. This implies that the action of inertia on $V_{\rbfn}$ must factor through $I_{L/K}$.  We also deduce that the action of inertia on $\vf$ must factor through $I_{L/K}$.  
\\ \\
\noindent
We define the functor $\mathcal{D}^{\Box, N=0}_{\vf}$ which assigns to every $A \in ob(\arwf)$ isomorphism classes of framed deformations of $\vf$ over $A$, whose restriction to $I_L$  factors though $I_{L/K}$.  $\mathcal{D}^{\Box, N=0}_{\vf}$ is a subfunctor of $\mathcal{D}^{\Box}_{\vf}$.

\begin{prop}
$\mathcal{D}^{\Box}_{\vf, N=0}$ is pro-represented by $\rbfn$.
\end{prop}
\begin{proof}
Let $A \in ob(\arwf)$ and let $V_A$ be a framed deformation of $\vf$.  By the universal property of $\rbf$ we know that this framed deformation must be induced by a local morphism $\phi:\rbf \rightarrow A$. Let $N_{univ} \in M_d(\rbf)$ be the universal monodromy operator.  Clearly the monodromy operator $N_{V_A} \in M_d(A)$ is the image of $N_{univ}$ under $\phi$.   By construction, $V_A$ factors through $I_{L/K}$ if and only if $N_{V_A}=0$.  Hence  the ideal generated by the entries of $N_{univ}$  is in the kernel of $\phi$ if and only if $N_{V_A}=0$.  The result follows immediately.
\end{proof}

\noindent
There is a more natural interpretation of this functor as follows:
\\ \\
\noindent
Let $I_L $ be the inertia subgroup of $L$.   Let $G_{un,L}: =G_K/I_L$.  It is naturally an extension of $\widehat{\mathbb{Z}}$ by the finite group $I_{L/K}$. 

Observe that $\vf$  naturally carries an action of $G_{un,L}$.  Let $A\in ob(\arwf)$.  If $V_A$ framed deformation $\vf$, then its restriction to $I_K$ factor through $I_{L/K}$ if and only if (as a $G_K$ representation) it factors through the quotient $G_{un,L}$. 

Let us define the functor $\mathcal{D}^{\Box}_{\vf, un, L}$,  which assigns to $A\in ob(\arwf)$ isomorphism classes of framed deformations of  $\vf$ (as a $G_{un, L}$ representation) over $A$.  By construction this functor is isomorphic to $\mathcal{D}^{\Box}_{\vf, N=0}$.  This reinterpretation is useful because it is of the form studied in \S2 setting $G= G_{un, L}$. In particular we observe that  because $G_{un, L}$ is a quotient of $G_K$ it satisfies the $p$-finiteness condition.  Hence  we may study the local geometry of $\xfn$ using deformation theoretic techniques.
\\ \\
\noindent
We now prove the main result of this section, which is a refinement of theorem 2.0.6 of \cite{Gee}.
\begin{thm}
If $\xfn$ is non-empty then it  is the union of formally smooth components each of dimension $d^2$.
\end{thm}
\begin{proof} Assume that $\xfn$ is non-empty.  Let $x\in \xfn$ have residue field $E/\qp$. Let $V_x$ be the associated representation of $G_K$.  By construction we know it factors through $G_{un,L}$.  Let $\mathcal{D}^{\Box, N=0}_{V_x} $denote the functor which assigns to every $A\in ob(\are)$ the set of isomorphism classes framed deformations of $V_x$ (as a $G_{un, L}$ representation) over $A$.    $\mathcal{D}^{\Box, N=0}_{V_x} $ is pro-represented by the complete local ring at $x$.  Hence we just need to show that  $\mathcal{D}^{\Box, N=0}_{V_x} $ is unobstructed and has tangent space of dimension $d^2$.

The obstruction to  $\mathcal{D}^{\Box, N=0}_{V_x} $ being smooth over is in $H^2(G_{un,L}, adV_x)$.  Observe that there is a short exact sequence 

$$ 0 \rightarrow I_{L/K} \rightarrow G_{un, L} \rightarrow \widehat{\mathbb{Z}} \rightarrow 0.$$
Where we have fixed an isomorphism $G_K/I_K \cong \widehat{\mathbb{Z}}$.  

A finite group acting on a vector space over as characteristic zero field has trivial cohomology in positive degree.  The higher inflation-restriction sequence coming from the Hoschild-Serre spectral sequence, tells us that 

$$H^2(G_{un,L}, adV_x) \cong  H^2(\widehat{\mathbb{Z}} , (adV_x)^{I_{L/K}}).$$

This latter groups is always trival, hence we deduce that the functor is unobstructed thus $\xfn$ is smooth at $x$. By lemma 1 we deduce that the dimension of the irreducible component containing $x$ is therefore equal to  $dim_E(H^1(G_{un,L}, adV_x)) +d^2 - dim_E((adV_x)^{G_K}) $.  The Hoschild-Serre spectral sequence again tells us that 

$$H^1(G_{un,L}, adV_x) \cong  H^1(\widehat{\mathbb{Z}} , (adV_x)^{I_{L/K}}).$$

It is well known that this latter space is equal to $H^0(\widehat{\mathbb{Z}} , (adV_x)^{I_{L/K}})$.  However,  this latter space is equal to $(adV_x)^{G_K}$ because $x \in \xfn$.  This completes the proof.
\end{proof}

\begin{thm}
If $L/K$ is tamely ramified then $\rbfn$ is formally smooth over $\wf$ of relative dimension $d^2$.
\end{thm}

\begin{proof}
By the proof of theorem 1 we just need to show that the deformation functor $\mathcal{D}^{\Box}_{\vf, N=0}$ is unobstructed in this case.  We know that $|I_{L/K}| =e(L/K)$ is coprime to $p$.  A standard result from the cohomology of finite groups tells us that the positive degree cohomology must be annihilated by $e(L/K)$.  Because $ad\vf$ is a vector space over $\f_p$ we deduce that $H^n(I_{L/K}, ad\vf) =0$ for $n>0$. Hence we may repeat the above argument to deduce that $H^2(G_{un,L}, ad\vf)=0$. The result follows.
\end{proof}

We do not know if this result remains true in the event of $L/K$ being wild.

\subsubsection{Local Raising Lowering Deformations}
As remarked in the introduction we will ultimately be concerned with $2$-dimensional $p$-adic global Galois representations whose restriction to a decomposition group at $l$ corresponds to a one dimensional representation under $\pi$, the local Langlands correspondence.  

Recall that if $V_E$ is a $2$-dimensional $E$-vector space, equipped with a continuous $E$-linear action of $G_K$ then under $\pi$, the associated representation of $GL_2(K)$ is one dimensional if and only if $V_E$ is a direct sum of two characters, one of which is a twist of the other by the $p$-adic cyclotomic character $\chi$.  The semi-simple mod $p$ representation associated to $V_E$ must therefore be the direct  sum of two characters, one of which is a twist of the other by the mod $p$ cyclotomic character $\bar{\chi}$.

Let  $\vf$ be a $2$-dimensional $\f$-vector space, equipped with a continuous representation of $G_K$ and a choice of basis $\beta_{\f}$.  Let $\lambda: G_K\rightarrow \f^*$ be a continuous character.  We define $\vf(\lambda):=\vf \otimes_{\f} \lambda$, the twist of $\vf$ by $\lambda$.
\begin{prop}The functors $\mathcal{D}^{\Box}_{\vf}$ and $\mathcal{D}_{\vf(\lambda)}^{\Box}$ are isomorphic.
\end{prop}
\begin{proof}
Let $\tilde{\lambda}: G_K \rightarrow W(\f)^*$ be the Teichmuiller lift of $\lambda$.  Let $A\in ob(\arwf)$.  We denote by $\tilde{\lambda}_A$ the canonical $A^*$-valued character induced by composition of $\tilde{\lambda}$ with the canonical inclusion $W(\f) \subset A$.   

Let $V_A$ be a deformation of $\vf$ over $A$.  By construction $V_A\otimes_A \tilde{\lambda}_A$ is a deformation of $\vf(\lambda)$ over $A$.  

Conversely if $V_A$ is a deformation of $\vf(\lambda)$, then $V_A\otimes_A \tilde{\lambda}_A^{-1}$ is a deformation of $\vf$ over $A$.  

This establishes a bijection between $\mathcal{D}_{\vf}(A)$ and $\mathcal{D}_{\vf(\lambda)}(A)$.  Given $A,B \in ob(\arwf)$ and a morphism $A\rightarrow B$, we have $\tilde{\lambda}_A\otimes_A B \cong \tilde{\lambda}_B$.  Hence we deduce that this gives a natural isomorphism $\mathcal{D}_{\vf} \cong \mathcal{D}_{\vf(\lambda)}$.  

Let $V_A$ be a deformation of $\vf$ over $A$, with a fixed $A$-basis $\beta_A$ lifting $\beta_{\f}$.  Then the image of $\beta_A$ under the natural morphism $V_A\rightarrow V_A\otimes_A \tilde{\lambda}_A$ gives an $A$-basis of $V_A\otimes_A \tilde{\lambda}_A$.  This gives rise to the natural  isomorphism of the framed deformation functors $\mathcal{D}^{\Box}_{\vf}\cong\mathcal{D}_{\vf(\lambda)}^{\Box}$
\end{proof}
Let  $\vf$ be a $2$-dimensional $\f$-vector space, equipped with a continuous representation of $G_K$ and a choice of basis $\beta_{\f}$.  Furthermore assume that 
\begin{enumerate}
\item $\vf (\bar{\chi}^{-1})^{G_K} \neq\{0\}$ 
\item $det(\vf)=\bar{\chi}$.
\end{enumerate}
These are precisely the conditions imposed in \S 2.6 of \cite{Kis5}.  In particular we have the the following result of Kisin:

\begin{prop}There is a closed subspace of $\mathfrak{X}_{\vf}^{\Box, \chi, sp}\subset \xf$  with the following two properties:
\begin{enumerate}
\item $\mathfrak{X}_{\vf}^{\Box, \chi, sp}$ is smooth of dimension of 3.
\item Let $x \in \xf$, have residue field $E/\qp$.  Then $x \in \mathfrak{X}_{\vf}^{\Box, \chi, sp}$ if and only if $V_x$, the representation associated to $x$, is an extension of $E$  by $E(\chi)$.
\end{enumerate}
\end{prop}
 \begin{proof} This is proposition 2.6.6 from \cite{Kis5}.
 \end{proof}
Note that the representations on $\mathfrak{X}_{\vf}^{\Box, \chi, sp}$ all have fixed determinant $\chi$.  We wish to remove this restriction.  

Let us define the subspace $\mathfrak{X}_{\vf}^{\Box, sp}\subset \xf$ by demanding that $x \in  \mathfrak{X}_{\vf}^{\Box, sp}$ if and only if the associated representation of $GL_2(K)$ under the local Langlands correspondence occurs as a subquotient of a reducible principal series.  More concretely, let $x \in \xf$ have residue field $E/\qp$.  Then $x\in \mathfrak{X}_{\vf}^{\Box, sp}$ if and only if $V_x$ is a twist of an extension of $E$ by $E(\chi)$.  Clearly $\mathfrak{X}_{\vf}^{\Box, \chi, sp}\subset \mathfrak{X}_{\vf}^{\Box, sp}$.

\begin{thm} $\mathfrak{X}_{\vf}^{\Box, sp}$ is formally smooth of dimension 4.
\end{thm}

\begin{proof}  By local Tate-Duality the universal deformation ring of the trivial representation of $G_K$ on a one dimensional $\f$-vector space has universal deformation ring isomorphic to the Iwasawa algebra over $W(\f)$.  In particular the associated deformation space, which we denote by $\mathfrak{X}$, is smooth of dimension 1.  

If $x \in \mathfrak{X}_{\vf}^{\Box, sp}$ then $V_x$ is a twist of a unique representation occuring in $\mathfrak{X}_{\vf}^{\Box, \chi, sp}$ by a unique character which is residually trivial.  We deduce that 
 $$\mathfrak{X}_{\vf}^{\Box, sp}\cong \mathfrak{X}\times \mathfrak{X}_{\vf}^{\Box, \chi, sp}.$$
 This in conjunction with proposition 5 yields the result.
\end{proof}
We should remark that by proposition 4 this result remains true after twisting $\vf$ be a character.  In particular,  given any $\vf$ such that $\mathfrak{X}_{\vf}^{\Box, sp}$ is non-empty, we know that it is formally smooth of dimension 4.

\subsection{Local Deformations at $p$}
In this section we review the theory of 2-dimensional trianguline deformation theory and recast Kisin's $h$-deformation functor in this language.
\\ \\
\noindent
Let  $\bc$ and $B_{dR}$ denote Fontaine's  crystalline and de Rham period rings (for an excellent survey of their construction see \cite{LB}).  Recall that $\bc$ is a topological $\qp$-algebra equipped with a continuous action of $\gqp$ and a continuous frobenius operator $\varphi$, which commutes with the Galois action.  Similarly, $B_{dR}$ is a topological $\qp$-algebra, which comes with a continuous action of $\gqp$ and a separated, exhaustive decreasing filtration.   $\bc$ is constructed from  $\bc^+ \subset B_{dR}^+ = Fil^0(B_{dR})$  by inverting the $p$-adic period $t \in \bc^+$.  $\bc$ naturally comes equipped with a separated, exhaustive decreasing filtration coming from the natural embedding into $B_{dR}$.  If $V$ is a finite dimensional $\qp$-vector space with a continuous $\qp$-linear action $\gqp$ we define the  functors $D_{cris}^+$, $D_{cris}$, $D_{dR}$ according to the usual recipe of Fontaine. For example

$$D_{cris}^+(V):=(\bc^+\otimes_{\qp}V)^{\gqp},$$
where $\gqp$ acts diagonally.  
\\ \\
\noindent
Let $E/\qp$ be a finite extension. Let $V_E$ be a finite dimensional $E$-vector space equipped with a continuous $E$-linear action of $\gqp$.  Following \cite{Kis4} we assume that $V_E$ has a non-trivial cristalline period, i.e. there exists  $\lambda \in E^{*}$ such that there is a non-zero vector

$$v_h \in D_{cris}^+(V_E)^{\varphi = \lambda}$$
Note that we are taking dual conventions to \cite{Kis4}, where the equivalent property of the dual representation is considered.  To clarify the link,  if we denote by $V_E^*$ the $E$-dual of $V_E$ then the existence of $v_h$ is equivalent to a non-zero $\gqp$-equivariant $E$-linear morphism
 $$h:  V_E^* \rightarrow (\bc^+ \otimes_{\qp} E)^{\varphi = \lambda}.$$

By theorem 6.3 of \cite{Kis4},  if $f$ is a finite slope,  overconvergent, cuspidal, $p$-adic eigenform defined over $E$, then the local representation attached to $f$ at $p$ satisfies this condition where $\lambda = a_p(f)$, and $a_p(f)$ is the eigenvalue of the $U_p$ operator.  Again note that we are implicitly dualising the global representation associated to $f$ in \cite{Kis4}.   This is the fundamental reason for introducing this somewhat technical looking definition.  A more conceptual approach using Fontaine's theory of $(\varphi, \Gamma)$-modules was introduced by Colmez (\cite{PC}), where, in the two dimensional case, he calls such representations (up to twist) trianguline.  This theory has been extensively developed both by Colmez, in the two dimensional case, and by Bellaiche and Chenevier in higher dimensions (\cite{BC}).  We now review this approach.

\subsection{($\varphi, \Gamma$)-modules and the Robba Ring}
Following  \cite{Kis4} we fix the convention that the $p$-adic cyclotomic character has Hodge-Tate weight 1 and Sen polynomial $X-1$. As in \cite{Chen} we normalise the local class fields theory reciprocity map to send a uniformiser to geometric Frobenius.   Under this choice of normalisation the the cyclotomic character corresponds to  $\chi$, the character on $\mathbb{Q}_p^*$ given by $\chi(x) = x|x|$.  The reader is cautioned  that our convention on the sign of  Hodge Tate weights is different than in \cite{Chen}  which we shall frequently cite.  As above let $E/\qp$ be a finite extension.  
\\ \\
\noindent
The \textit{Robba} ring with coefficients in $E$ is the $E$-algebra $\mathcal{R}_E$ of powers series

$$f(z) = \sum_{n \in \mathbb{Z}} a_n(z-1)^n, \; a_n \in E$$
converging on some annulus of $\mathbb{C}_p$ of the form $r(f) \leq |z-1|<1$, equipped with its natural $E$-algebra toplology.  $\mathcal{R}_E$ is naturally equipped with commuting $E$-linear, continuous actions of $\varphi$ and the group $\Gamma:=\mathbb{Z}_p^*$ defined by $$\varphi(f)(z) = f(z^p), \; \;\;\;\;\gamma(f)(z) = f(z^{\gamma}).$$  Note that $\mathcal{R}_E = \mathcal{R}_{\qp}\otimes_{\qp}E$. Similarly, if $A \in ob(\are)$ we define $\mathcal{R}_A := \mathcal{R}_{\qp}\otimes_{\qp}A$.  In the case when $A =E$, the \textit{Robba} ring over $E$ is a Bezout domain.
An important element of $\mathcal{R}_{\qp}$ (and consequently any $\mathcal{R}_A$) is
$$t = log(z):= \sum_{n \geq 1} (-1)^{n+1}\frac{(z-1)^n}{n}.$$
It is important to observe that

$$\varphi(t) =pt, \;\; \gamma(t) = \gamma t, \; \forall \gamma \in \Gamma.$$
\begin{defin} let $A\in ob(\are)$. A $(\varphi, \Gamma)$-module over $\mathcal{R}_A$ is a finitely generated $\mathcal{R}_A$-module $D$ which is free over $\mathcal{R}_{\qp}$ and eqiupped with commuting, $\mathcal{R}_A$-semilinear, contiuous actions of $\varphi$ and $\Gamma$, and such that $\mathcal{R}_{\qp}\varphi(D)=D$.
\end{defin}
\noindent
Let $A \in ob(\are)$.   Let $D$ be a $(\varphi, \Gamma)$-module over $\mathcal{R}_A$.  Let $D$ be of rank $d \in \mathbb{N}$ over $\mathcal{R}_{\qp}$. By work of Kedlaya (\cite{Ked}) we may associate to $D$  a sequence of rational numbers $s_1 \leq \cdots \leq s_d$ called the slopes of $D$.  We say that $D$ is \textit{etale} if all the slopes are 0.  
By work of Fontaine, Cherbonnier-Colmez and Kedlaya (see proposition 2.7 of \cite{PCT}) we have:

\begin{prop}There is a $\otimes$-equivalence of categories between $A$-representations of $\gqp$ and etale $(\varphi, \Gamma)$-module over $\mathcal{R}_A$.  
\end{prop}

By lemma 2.2.7 of \cite{BC} if $V_A$ is a free, rank $n$, $A$-module equipped with a continuous $A$-linear action of $\gqp$ then the associated etale $(\varphi, \Gamma)$-module is free of rank $n$ over $\mathcal{R}_A$.

\subsection{  $(\varphi, \Gamma)$-modules of rank one  }

Let $A \in \are$.  Let $D$ be a $(\varphi, \Gamma)$-module over $\mathcal{R}_A$.  We say that $D$ is of rank one if it is free of rank one over $\mathcal{R}_A$.  Let $\delta: \mathbb{Q}_p^* \longrightarrow A^*$ be a continuous character.   Following Colmez (\cite{PC}) we may associate to $\delta$ a rank one $(\varphi, \Gamma)$-module $\mathcal{R}_A(\delta)$ as follows:  we equip $\mathcal{R}_A$ with the following semilinear actions
$$\varphi(1) = \delta(p)1,\;\;\;\;\; \;\; \gamma(1) = \delta(\gamma)1 \;\;\; \forall \gamma \in \Gamma.$$

By proposition 4.2 of \cite{Chen}, we know that any rank one $(\varphi, \Gamma)$-module over $\mathcal{R}_A$, arises as above for a unique such $\delta$.  We remark that if $\bar{\delta}$ is the reduction of this character modulo the maximal ideal then the $\mathcal{R}_A(\delta)$ is etale if and only if $\bar{\delta}(p) \in \mathcal{O}_E^*$.

\begin{defin} Let $A \in \are$.  Let $D$ be a $(\varphi, \Gamma)$-module over $\mathcal{R}_A$, which is free of rank 2.  We say that $D$ is \textit{trianguline}  if $D$ is an extension of two rank one $(\varphi, \Gamma)$-modules over $\mathcal{R}_A$.  A triangulation of $D$ is a choice of a rank one $(\varphi, \Gamma)$-modules $D'\subset D$ such that $D/D'$ is rank one.  More precisely $D$ is trianguline if there exist two continuous characters $\delta_1, \delta_2 : \mathbb{Q}_p^* \longrightarrow A^*$ such that there is a short exact sequence (in the category of $(\varphi, \Gamma)$-modules):

$$0 \rightarrow \mathcal{R}_A(\delta_1) \longrightarrow D_E \longrightarrow \mathcal{R}_A(\delta_2) \longrightarrow 0.$$ 
\end{defin}
\noindent
There is a natural generalisation of this definition to higher dimension (\cite{BC}).

Let $V_A$ be a free $A$-module equipped with a continuous $A$-linear action of $\gqp$.  By proposition 6 we may associate to $V_A$ a free $(\varphi, \Gamma)$-module over $\mathcal{R}_A$, which we denote $D_A$.  We say that $V_A$ is trianguline if $D_A$ is trianguline.  We can detect whether $V_A$ is trianguline directly by using the following crucial result of Colmez and Chenevier.

\begin{prop}
Let $A \in ob(\are)$.  Let $D_A$ be an etale  $(\varphi, \Gamma)$-module over $\mathcal{R}_A$, which is free of rank 2 corresponding to $V_A$ as above.    Suppose that for $\lambda \in A^*$ there exists a free $A$-module of rank one $A.v \subset D_{cris}(V_A)^{\varphi = \lambda}$.  Let $s \in \mathbb{Z}$ be the minimal integer such that $v \notin Fil^{s+1}D_{cris}(V_A)$ then $D_A$ is an extension of $\ra(\delta_2)$ by $\ra(\delta_1)$, where $\delta_1(p) = \lambda p^{-s}$ and $(\delta_1)|_{\Gamma} = \chi^{-s}$, and $\delta_2 = det (V_A)\delta_1^{-1}$.  In particular $D_A$ is trianguline.
\end{prop}
\begin{proof}This is proposition 4.6 of \cite{Chen}.
\end{proof} 

\begin{prop}Let $A \in \are$. Let $\delta_1, \delta_2: \qp^*\rightarrow A^*$ be two continuous characters.  Let $\bar{\delta}_1, \bar{\delta}_2: \qp^*\rightarrow E^*$ be their respective reductions modulo the maximal ideal of $A$.  Suppose that $\bar{\delta}_1\bar{\delta}_2^{-1} \neq x^{-i}, \chi x^i$ for $i \geq 0$.  Then $Ext_{(\varphi, \Gamma)}(\mathcal{R}_A(\delta_2), \mathcal{R}_A(\delta_1))$ is a free $A$-module of rank 1.
\end{prop}
\begin{proof}This is proposition 4.3 of \cite{Chen}
\end{proof}

\subsection{Trianguline Deformations and Kisin's $h$-deformations} 
The aim of this section is to compare the trianguline deformation functor of \cite{BC} and the $h$-deformation functor introduced in \cite{Kis4}.   In favourable circumstances the two will turn out to be isomorphic.
\\ \\
\noindent
For the rest of the section let $V_E$ be a 2-dimensional $E$-vector space with continuous $E$-linear action of $\gqp$.  We say that $V_E$ satisfies $(\dagger)$ if

\begin{enumerate}
\item $V_E$ is absolutely irreducible.
\item  There exists $\lambda \in E^*$ such that there exists a non-zero period
$$v_h \in D_{cris}^+(V_E)^{\varphi = \lambda}.$$
\item $V_E$ is not cristalline (hence $v_h$ is unique up to a multiple of $E^*$).
\item The generalised Hodge-Tate weights of $V_E$ are $0$ and $k\notin -\mathbb{N} \cup \{0\}$.
\end{enumerate}
$(\dagger)$ is exactly the type of condition introduced in \cite{Kis4}.   We wish to recast it in the language of $(\varphi, \Gamma)$-modules.
\\ \\
\noindent
Let $D_E$, denote the $(\varphi, \Gamma)$-module associated to $V_E$.
\begin{prop} Assume that $V_E$ satisfies $(\dagger)$. Then $D_E$ is an extension  of the following form:
$$0 \rightarrow \mathcal{R}_E(\delta_1) \longrightarrow D_E \longrightarrow \mathcal{R}_E(\delta_2) \longrightarrow 0,$$ 
where  $\delta_1$ is trivial on $\mathbb{Z}_p^*$. Moreover, $\delta_1$ and $\delta_2$ are unique and the extension is not split in the category of $(\varphi, \Gamma)$-modules over $\mathcal{R}_E$. Moreover, the collection of such extensions of $D_E$ forms a torsor under $E^*$.  In particular $D_E$ is trianguline with a unique triangulation.
\end{prop}
\begin{proof}This is just an application of proposition 7 observing that  $D_{cris}^+(V_E) \subset Fil_0( D_{cris}(V_E))$ and the jumps in the filtration must occur at negative integers by $(4)$ of $\dagger$.  The assumption that $V_E$ is not cristalline together with the  part $(iii)$ of theorem 0.5 of \cite{PC}  tells us that the characters $\delta_1$ and $\delta_2$ are unique.  The extension is not split because this would contradict part $(1)$ of $(\dagger)$.  

We must finally show that the collection of such extensions is a torsor under $E^*$.  This is equivalent to showing that $Ext_{(\varphi, \Gamma)}(\mathcal{R}_E(\delta_2), \mathcal{R}_E(\delta_1))$ is a one dimensional $E$-vector space.  The  proposition 8 this is precisely when $\delta_1\delta_2^{-1} \neq x^{-i}, \chi x^i$ for $i \geq 0$.  The second condition cannot hold because of property $(4)$ of $(\dagger)$.  Thus, we need to overrule the situation where $\delta_1\delta_2^{-1} = x^{-i}$ for $i>0$.  By remark 3.4 of \cite{Chen} we know that $\delta_1\delta_2(p) \in \mathcal{O}_E^*$ and $\delta_1(p) \in \mathcal{O}_E$.  By property $(1)$ of $(\dagger)$ we know, furthermore, that $\delta_1(p) \notin \mathcal{O}_E^*$. Thus, $\delta_2(p) \notin \mathcal{O}_E$.  If $\delta_1\delta_2^{-1} = x^{-i}$ for some $i>0$, then $\delta_1(p) = p^{-i}\delta_2(p)$.  This second term is not integral, which is a contradiction. Thus, $\delta_1\delta_2^{-1} \neq x^{-i}$ for $i>0$. The result follows.
\end{proof}
\noindent
From now on assume that $V_E$ satisfies $(\dagger)$.  Fix a non-zero cristalline period $$v_h \in D_{cris}^+(V_E)^{\varphi = \lambda}.$$Following \cite{Kis4} we define  the $h$-deformation functor,  $\mathcal{D}^{h, \varphi}_{V_E}$ on $\are$, which assigns to any $A \in ob(\are)$ the set of isomorphism classes $V_A$ of deformations of  $V_E$ such that there exists $\tilde{\lambda} \in A^*$ lifting $\lambda$ and a non-zero period  $$\tilde{v}_h \in \mathcal{D}_{cris}^+(V_A)^{\varphi = \tilde{\lambda}}$$
lifting $v_h$.  By proposition 8.12 of \cite{Kis4},  $\tilde{\lambda}$ is uniquely determined by $\lambda$ and $\tilde{v}_h$ up to a multiple of $A^*$.
\\ \\
\noindent
Following \cite{BC}, we define the trianguline deformation functor $\mathcal{D}_{V_E}^{tr}$, in the language of $(\varphi, \Gamma)$-modules.  To any $A\in ob(\are)$,  $\mathcal{D}_{V_E}^{tr}(A)$ is the set of isomorphism classes of triples $(D_A, \pi, T)$, where 
\begin{enumerate}
\item $D_A$ is a free $(\varphi,\Gamma)$-module of rank 2 over $\mathcal{R}_A$.
\item $\pi: D_A \otimes_A E \cong D_E$ is an isomorphism of $(\varphi, \Gamma)$-modules over $\mathcal{R}_E$, i.e. $D_A$ is a deformation of $D_E$.
\item $T$ is a triangulation of $D_A$ associated to an extension of  the form:

$$\xymatrix{ 0 \ar[r] &\mathcal{R}_A(\tilde{\delta}_1) \ar[r] &D_A \ar[r] &\mathcal{R}_A(\tilde{\delta_2}) \ar[r] &0
},$$
where $\tilde{\delta}_1, \tilde{\delta}_2: \qp^*\rightarrow A^*$ are two continuous characters such that $\tilde{\delta}_1|_{\mathbb{Z}_p^*} = 1$ and after tensoring over $A$ with $E$ and applying $\pi$ we recover  $D_E$ together with its unique triangulation.  
\end{enumerate}
Observe that this implies that in this case $\tilde{\delta}_1$ and  $\tilde{\delta}_2$ reduce to $\delta_1$ and $\delta_2$ modulo the maximal ideal of $A$. This definition is slightly different than the one found in \cite{BC}.  We are restricting the behaviour of $\tilde{\delta}_1$.  
\noindent 
\begin{prop}
Let $V_E$ satisfy $(\dagger)$.   Then  $\mathcal{D}^{h, \varphi}_{V_E}$ and $\mathcal{D}_{V_E}^{tr}$ are isomorphic.
\end{prop}
\begin{proof}   Let $D_E$ be the $(\varphi, \Gamma)$-module over $\mathcal{R}_E$ associated to $V_E$.  By proposition 2.3.13 of \cite{BC} we know that the defomation functor of $V_E$ and the deformation functor of $D_E$ are isomorphic.   Thus, for $A \in ob(\are)$, we must show that the existence of a non -zero crystalline period of $V_A$, lifting $v_h$, gives rise to a unique triangulation on $D_A$ with the desired  properties, and vice versa.
\\ \\
\noindent
Let $A\in ob(\are)$.  Choose a  deformation $V_A$ of $V_E$ such that there exists a lifting 
$$\tilde{v}_h \in \mathcal{D}_{cris}^+(V_A)^{\varphi = \tilde{\lambda}}.$$
Let $D_A$ be the etale $(\varphi, \Gamma)$-module associated to $V_A$.  By proposition 2.3.13 of \cite{BC}, This is a deformation of $D_E$, the $(\varphi, \Gamma)$-module associated to $V_E$.  

The existence of $\tilde{v}_h$ together with the above proposition implies that $D_A$ is trianguline.  We know any triangulation must reduce to the unique triangulation on $D_E$.  Condition $(4)$ of $(\dagger)$ implies that $\delta_1\delta_2^{-1} \neq x^i$ for all $i>0$.  Hence by proposition 2.3.6 of \cite{BC} we know that this triangulation is unique. This also follows for the above proposition 8.  If the rank one submodule has associated character $\tilde{\delta_1}$, then we know that it must reduce to $\delta_1$ modulo the maximal ideal.  However by above proposition we know  that  $\tilde{\delta_1}|_{\Gamma} = \chi^{-s}$ for some $ s \geq 0$.  Because $\delta_1|_{\Gamma} = 1$, the only way these two conditions can hold is if $\tilde{\delta_1}|_{\Gamma} = 1$.  Hence the triangulation on $D_A$ is of the desired form.
\\ \\
\noindent
Conversely let us assume that $D_A$ a free $(\varphi,\Gamma)$-module of rank 2 over $\mathcal{R}_A$ deforming $D_E$, together with a triangulation of the desired form.  Again by proposition 2.3.13 of \cite{BC} we know that $D_A$ is etale and corresponds to $V_A$, a deformation of $V_E$.   

Let $\tilde{v_h} \in D_A$ be a basis for the rank one submodule  $\mathcal{R}_A(\tilde{\delta}_1)$ such that $\Gamma$ acts trivially on $\tilde{\vi}_h$ and $\varphi$ acts by some $\tilde{\lambda}$, a lift of $\lambda$. Hence

$$ \tilde{v}_h  \in D_A^{\Gamma=1, \varphi = \tilde{\lambda}}.$$
A fundamental result of Berger (\cite{LBD}) tells us that there is a natural filtration and Frobenius preserving $A$-module isomorphism:

$$\mathcal{D}_{cris}(V_A) \cong ([1 / t] D_A)^{\Gamma =1}$$
Hence (under this isomorphism) $\tilde{v}_h \in \mathcal{D}_{cris}(V_A) ^{\varphi = \tilde{\lambda}}$. By lemma 2.4.2 of \cite{BC} we know that 
$$\tilde{v}_h \in Fil_0(\mathcal{D}_{cris}(D_A)^{\varphi = \tilde{\lambda}})$$
By lemma 3.2 of \cite{Kis4}  we know that 
$$Fil_0(\mathcal{D}_{cris}(D_A)^{\varphi = \tilde{\lambda}}) =  \mathcal{D}_{cris}^+(V_A)^{\varphi = \tilde{\lambda}}.$$
Hence 
 $$\tilde{v}_h \in   \mathcal{D}_{cris}^+(V_A)^{\varphi = \tilde{\lambda}}$$
Finally we may scale by an element of $A^*$, to ensure that $\tilde{v}_h$ is a lift of $v_h$.  This establishes a  bijection between  $\mathcal{D}^{h, \varphi}_{V_E}(A)$ and $\mathcal{D}_{V_E}^{tr}(A)$.  This establishes a natural isomorphism because the the fact that lifts of triangulations are necessarily unique by proposition 2.3.6 of \cite{BC}.
\end{proof}
This result allows us to use the techniques of $(\varphi, \Gamma)$-theory to determine the properties of this functor.

\begin{prop}
Assume that $V_E$ satisfies $(\dagger)$. Then the  functor  $\mathcal{D}^{h, \varphi}_{V_E}$ is pro-representable and formally smooth.
\end{prop}
\begin{proof}By proposition 9,  $\mathcal{D}^{h, \varphi}_{V_E}$ is isomorphic to $\mathcal{D}_{V_E}^{tr}$. 
The assumptions in $(\dagger)$  allows us to apply proposition 2.3.9 and 2.3.10 of \cite{BC}.  The result follows immediately.
\end{proof}

\begin{prop}
Let us assume that $V_E$ satisfies $(\dagger)$. Then the functor $\mathcal{D}^{h,\varphi}_{V_E}$ is pro-representable by a formally smooth complete local Noetherian ring  of dimension 3, which we denote $R^{h,\varphi}_{V_E} \in \are$
\end{prop}

\begin{proof}
By the previous proposition we know that we need only to determine the  dimension of the tangent space of $\mathcal{D}^{h,\varphi}_{V_E}$.  Let $E[\epsilon]: = E[X]/(X^2)$ be the dual numbers.  We wish to determine $dim_E(\mathcal{D}^{h,\varphi}_{V_E}(E[\epsilon]))$.  Let $D_{E[\epsilon]}$ be a deformation of $D_E$.  By $(\dagger)$ and proposition 8 we know that any triangulation on $D_{E[\epsilon]}$ lifting the unique one on $D_E$ is necessarily unique.  Conversely, if  $\tilde{\delta}_1, \tilde{\delta}_2: \qp^* \rightarrow E[\epsilon]^*$  are two continuous characters lifting $\delta_1$ and $\delta_2$, then by proposition 8 we know that there is a unique (up to isomorphism) trianguline deformation $D_{E[\epsilon]}$ of $D_E$ with weights characters $\tilde{\delta}_1$ and $ \tilde{\delta}_2$.

Thus the problem is reduced to determining the acceptable lifts of $\delta_1$ and $\delta_2$.  There is no restriction on the $\tilde{\delta}_2$ lifting $\delta_2$ and the space of such lifts is naturally a 2 dimensional $E$-vector space.  The restriction  that $\tilde{\delta}_1$ is trivial on $\mathbb{Z}_p^*$ means that the space of appropriate lifts has $E$-dimension 1.  We deduce that $dim_E(\mathcal{D}^{h,\varphi}_{V_E}(E[\epsilon])) =3$. The result follows.
\end{proof}
It will be convenient to consider a framed version of $\mathcal{D}^{h, \varphi}_{V_E}$.  We will denote this new functor by $\mathcal{D}_{V_E}^{\Box, h}:=\mathcal{D}^{\Box}_{V_E}\times_{\mathcal{D}_{V_E}}\mathcal{D}^{h, \varphi}_{V_E}$.  We have removed $\varphi$ merely to simplify the notation.  The forgetful morphism $\mathcal{D}_{V_E}^{\Box, h} \rightarrow \mathcal{D}^{h, \varphi}_{V_E}$ is formally smooth of relative dimension 3. Hence we have :
\begin{cor}Let $V_E$ satisfy $(\dagger)$. then the functor $\mathcal{D}_{V_E}^{\Box, h}$ is pro-representable by a formally smooth complete local Noetherian ring  of dimension 6, which we denote $R^{\Box, h}_{V_E} \in \are$.
\end{cor}
\begin{proof}
The result is immediate after observing, by proposition 12, that $\mathcal{D}^{h,\varphi}_{V_E}$ is pro-represented by a formally smooth complete local noetherian ring of dimension 3 and that $\mathcal{D}_{V_E}^{\Box, h} \rightarrow \mathcal{D}^{h, \varphi}_{V_E}$ is formally smooth of relative dimension 3.
\end{proof}

\section{Global Deformations}
\subsection{Presenting Global Deformation Rings Over Local Ones}
We will ultimately be studying global deformation rings of deformations with prescribed behaviour both at $l$ and $p$, to give us information about the local geometry of deformation spaces.  Hence in this section we develop the theory of presenting global deformation rings over local deformation rings in the characteristic zero case.  We follow \cite{Kis2}, very closely, where the unequal characteristic case is considered.
\\ \\
Let $F$ be a number field and $S$ a finite set of places of $F$ containing all those which divide the prime $p$.  Fix an algebraic closure $\bar{F}$ of $F$ and denote $F_S \subset \bar{F}$ the maximal extension of $F$ unramified outside of $S$.  Write $G_{F,S} = Gal(\bar{F}/F)$.

Let $\Sigma \subset S$ and fix an algebraic closure $\bar{F_{\textit{v}}}$ of $F_{\textit{v}}$ for each $\textit{v} \in \Sigma$.  Write $\gv = Gal(\bar{\fv}/\fv)$. We also fix embeddings $\bar{F} \subset \bar{F_{\textit{v}}}$ inducing the inclusion $\gv \subset G_{F, S}$ for each $\textit{v} \in \Sigma$.  All Galois groups we have described satisfy the $p$-finiteness condition introduced in \S 2.

Let $E/ \mathbb{Q}_p$ be a finite extension and $V_E$ a finite dimensional $E$-vector space equipped with a continuous action of $G_{F, S}$.    Let $d = dim_E(V_E)$.

For each $\textit{v} \in \Sigma$ fix a basis $\beta_{\textit{v}}$ of $V_E$.  We denote by $\mathcal{D}^{\Box}_{\textit{v}}$ the framed deformation functor of $(V_E |_{\gv}, \beta_{\textit{v}})$.  We denote the universal framed deformation ring by $\rbv$.  Recall that it is a complete local Noetherian ring with residue field $E$.  We set $\rbs = \hat{\otimes}_{\textit{v}\in \Sigma}\rbv$ and we denote by $\mbs$ the maximal ideal of $\rbs$.

Similarly, we define the functor  $\mathcal{D}^{\Box}_{F,S}$ which to every $A \in ob(\are)$ assigns the set of isomorphism class of tuples $(V_A, \{ \beta_{\textit{v}, A} \}_{\textit{v} \in \Sigma})$ where $V_A$ is a deformation of $V_E$ as a $G_{F, S}$ representation and $\beta_{\textit{v}, A}$ is an $A$-basis.  This functor is pro-represented by a complete local Noetherian ring which we denote by $\rbfs$.  Again we denote the maximal ideal of $\rbfs$ by $\mbfs$.

The functoriality of our construction together with the Yoneda lemma ensures that there is a natural morphism $\rbs \rightarrow \rbfs$.  Hence $\rbfs$ naturally has the structure of a $\rbs$-algebra.  

For $i \in \{0, 1, 2\}$ we define $h^i_{\Sigma}$ to be the dimension of the kernel to the natural map of $E$-vector spaces:

$$\gamma^i: H^i(G_{F, S}, adV_E) \rightarrow \prod_{\textit{v}\in \Sigma}H^i(\gv, adV_E).$$

\begin{prop}
Let 

$$ \eta : \mbs/ \mbss \rightarrow \mbfs / \mbfss$$

be the map on tangent spaces induced by the natural map $\rbs \rightarrow \rbfs$.   Then $\rbfs$ is a quotient of a power series ring over $\rbs$ in $dim_E(coker\eta)$ variables by at most $dim_E(ker\eta) + h^2_{\Sigma}$ relations.
\end{prop}
\begin{proof}
We follow the proof of proposition 4.1.4 of \cite{Kis2} very closely.  Our situtation is simpler as all rings are equicharacteristic zero.

The number of generators of $\rbfs$ over $\rbs$ is equal to the $E$-dimension of the tangent space of $\rbfs / \mbs$.  The tangent space of this ring is naturally equal to $\mbfs / (\mbfss, \mbs) = coker(\eta)$.  This proves the result about the number of generators.  For ease of notation let $h = dim_E(coker(\eta))$.  Hence there is a surjection
$$\tilde{R} := \rbs [[x_1,  . . . x_h]] \rightarrow \rbfs,$$
which by construction induces a surjection on tangent spaces whose kernel is naturally isomorphic to $ker(\eta)$.   Let $J$ be the kernel of this surjection.  We wish to bound the number of generators of $J$.  Write $\tilde{m}$ for the maximal ideal of $\tilde{R}$. Nakayama's lemma tells us that the minimal number of generators of $J$ is 
equal to the $dim_E(J/\tilde{m}J)$.  Let 

$$\rho: G_{F,S} \rightarrow GL_d(\rbfs)$$
denote the universal framed deformation and consider a continuous set theoretic lifting 

$$\tilde{\rho} : G_{F,S} \rightarrow GL_d(\tilde{R}/ \tilde{m}J)$$
of $\rho$. Such a lifting is always possible.  Define a 2-cocycle

$$c: G^2_{F,S} \rightarrow J/\tilde{m}J \otimes_E adV_E; \;\;\; c(g_1, g_2) = \tilde{\rho}(g_1g_2)\tilde{\rho}(g_2)^{-1}\tilde{\rho}(g_1)^{-1}.$$
Here we are naturally embedding the kernel of $GL_d(\tilde{R}/ \tilde{m}J) \rightarrow GL_d(\rbfs)$ in $J/\tilde{m}J \otimes_E adV_E$.

The class $[c]$ of $c$ in $H^2(G_{F,S}, adV_E)\otimes_E J/\tilde{m}J$ depends only on $\rho$ and not $\tilde{\rho}$ .  It vanishes if and only if $\tilde{\rho}$ can be chosen to be a homomorphism.  By construction it it clear that for each $\textit{v} \in \Sigma$, $\rho|_{\gv}$ may be lifted to $GL_d(\rbv)$ and hence $GL_d(\tilde{R})$.  We deduce that the image of $c$ in $H^2(\gv, adV_E)\otimes_E J/\tilde{m}J = 0$
Hence $[c] \in ker(\gamma^2) \otimes_E J/\tilde{m}J$.  Hence if $(J/\tilde{m}J)^*$ is the $E$-dual of $J/\tilde{m}J$ we get the natural map

$$\delta: (J/\tilde{m}J)^* \rightarrow ker(\gamma^2); \;\;\; u \rightarrow \langle [c], u\rangle.$$
For ease of notation let us write $I = ker\eta$.  Observe that $(J/\tilde{m}J)$ surjects onto $I \subset \tilde{m}/\tilde{m}^2$ and hence we get an inclusion $I^* \subset (J/\tilde{m}J)^*$.  We claim that $I^*$ contains the kernel of $\delta$.

Suppose that $0 \neq u \in (J/\tilde{m}J)^*$ is in $ker\delta$.  Let $\tilde{R}_u$ be the push-out of $\tilde{R}/\tilde{m}J$ by $u$.  Hence $\rbfs \cong \tilde{R_u}/ I_u$ where $I_u \subset \tilde{R}_u$ is an ideal of square zero, which is isomorphic to $E$ as an $\tilde{R_u}$-module.  Since $u \in ker\delta$, we know that $\rho$ lifts to a representation $\rho_u : G_{F,S} \rightarrow GL_d(\tilde{R}_u)$.  Hence the natural map $\tilde{R}_u \rightarrow \rbfs$ has a section by the universal property of $\rbfs$.  Hence $\tilde{R}_u \cong \rbfs \oplus I_u$.  In particular the map  $\tilde{R}_u \rightarrow \rbfs$ does not induce an isomorphism on tangent.  We deduce that the composite

$$ker( J/\tilde{m}J \rightarrow I) \rightarrow J/\tilde{m}J \rightarrow I_u$$
is not surjective, and hence is zero.  We conclude that $u$ factors through $I$.
By rank nullity we deduce that 

$$dim_E(J/\tilde{m}J)^* \leq dim_EI + dim_E(ker\gamma^2) = dim_E(ker\eta)+h_{\Sigma}^2.$$

\end{proof}

\begin{prop}
Suppose that $\Sigma$ contains all places of $F$ dividing $p$ and $S$ contains all places dividing $\infty$.  Let us also assume that the map

$$\theta : H^0(G_{F, S}, (adV_E)^*(1)) \rightarrow \prod_{\textit{v} | \infty} \hat{H}^0(\gv,  (adV_E)^*(1)) \times  \prod_{\textit{v} \in (S \backslash \Sigma)_f} H^0(\gv,  (adV_E)^*(1))$$
is injective, where $(S \backslash \Sigma)_f$ denotes the finite places and $\hat{H}^0(\gv,  (adV_E)^*(1)) $ is the usual cohomology group modulo the subgroup of norms. We note that this condition is automatically satisfied if $(S \backslash \Sigma)_f\neq \emptyset$. 

Then there exists a positive integer $r$ such that 
$$\rbfs = \rbs[[x_1, ..., x_r]] / (f_1, ..., f_{r+s}).$$
Where $s = \sum_{\textit{v} | \infty, \textit{v} \notin \Sigma} dim_E(H^0(\gv, adV_E)).$
\end{prop}
\begin{proof}
We have the following series of natural equalities:

\begin{align*}
 dim_E(\mbfs/ \mbfss) &= dim_E(\mathcal{D}^{\Box}_{F,S}(E[\epsilon])) \\ 
 &=d^2 |\Sigma | + dim_E(H^1(G_{F, S}, adV_E)) - dim_E(H^0(G_{F, S}, adV_E)).
\end{align*}
Similarly we have 

\begin{align*}
dim_E(\mbs/ \mbss) &= \sum_{\textit{v} \in \Sigma} dim_E(\mathcal{D}^{\Box}_{\textit{v}}(E[\epsilon])) \\
&= d^2|\Sigma |  + \sum_{\textit{v} \in \Sigma}[dim_E(H^1(\gv, adV_E)) - dim_E(H^0(\gv, adV_E)].
\end{align*}
Local Tate duality, the final three terms of the Poitou-Tate sequence and the injectivity of $\theta$ imply that we have the equality

$$h^2_{\Sigma} = dim_E(H^2(G_{F, S}, adV_E)) - \sum_{\textit{v} \in \Sigma} dim_E(H^2(\gv, adV_E))$$

By proposition 13 we know that we may convert the information about the number of generators and relations into statements about local and global Euler characteristics.  If $\eta$ is the map defined in proposition 13 then we observe that 

\begin{align*}
s &=  dim_E(ker\eta) + h^2_{\Sigma} -dim_E(coker\eta)  \\
&= dim_E(\mbs/ \mbss) - dim_E(\mbfs/ \mbfss) \\
&\;\;\;\; + dim_E(H^2(G_{F, S}, adV_E)) - \sum_{\textit{v} \in \Sigma} dim_E(H^2(\gv, adV_E)) \\
&= \chi(G_{F, S}, adV_E) - \sum_{\textit{v} \in \Sigma} \chi(\gv, adV_E)
\end{align*}
Now we may apply Tate's results on local and global euler characteristic formulae.  Hence we get the equality

$$s = \sum_{\textit{v} | \infty, \textit{v} \notin \Sigma} dim_E(H^0(\gv, adV_E)).$$
\end{proof}

\subsection{Global Applications}

Let $\f$ be a finite field of characteristic $p$.  Let $S$ be a finite set of places of $\mathbb{Q}$ containing $p$ and $\infty$.  Fix an embedding $\bar{\q}\subset \bar{\q}_p$.  Let $\vf$ be a two dimensional $\f$-vector space equipped with a continuous, odd action of $G_{\mathbb{Q}, S}$.  Furthermore, assume that $end_{\f[G_{\mathbb{Q}, S}]}\vf = \f$.  This ensures that the deformation functor $\mathcal{D}_{\vf}$ is pro-representable with universal noetherian deformation ring $\rf\in ob(\arwf)$  Let $\xff$ denote the associated universal deformation space of $\vf$. Following \S11 of \cite{Kis4}  we may construct a Zariski closed subspace $X_{fs} \subset \xff \times \mathbb{G}_m$, which we may view as a Galois theoretic analogue of the Coleman-Mazur eigencurve (\cite{CM}, \cite{KB}) of the appropriate tame level.  The construction of this subspace is rather technical and we refer the reader to \cite{Kis4} for a precise definition.   The following crucial property of $X_{fs}$ allows one to use Galois deformation theoretic techniques to study its local geometry:

\begin{thm}
Let $(x, \lambda) \in X_{fs}$ have residue field $E$.  Let $V_x$ denote the induced $E$-vector space together with it's natural action of $G_{\mathbb{Q}, S}$. Let $V^p_x$ denote the restriction of this representation to $\gqp$.  Assume that the generalised Hodge-Tate weights of $V_x^p$ are 0 and $k \notin -\mathbb{N} \cup \{0\}$.  Then there exists a non-zero vector

$$v_h \in D_{cris}^+(V_x^p)^{\varphi = \lambda}$$
Furthermore,  the complete local ring $\hat{\mathcal{O}}_{X_{fs}, (x, \lambda)}\in \are$ pro-represents the functor

$$\mathcal{D}_{V_x} \times_{\mathcal{D}_{V_x^p}} \mathcal{D}_{V_x^p}^{h, \varphi}.$$
\end{thm}
\begin{proof}
First note that we are taking the dual conventions to \cite{Kis4}.  Hence the existence of $v_h$ is equivalent to the existence of a non-zero, $\gqp$-equivariant $E$-linear morphism 
 $$h:  V_E^* \rightarrow (\bc^+ \otimes_{\qp} E)^{\varphi = \lambda}.$$
 
The result then follows by a  combination of theorems 11.2 and 11.3 of \cite{Kis4}, after observing that condition (11.2)(1) is the same as our restriction on the generalised Hodge Tate weights of $V_x^p$.
\end{proof}
\noindent
Let $(x, \lambda) \in X_{fs}$ .  We say that $(x, \lambda)$ satisfies $(\ddagger)$  if
\begin{enumerate}
\item $V_x$ and $V_x^p$ are absolutely irreducible.
\item $V_x^p$ has generalised Hodge-Tate weights $0$ and $k\notin -\mathbb{N}\cup \{0\}$.
\item $V_x^p$ is not cristalline.

\end{enumerate}
We should observe that under these assumptions the associated trianguline $(\varphi, \Gamma)$-module satisfies:  $$0 \rightarrow \mathcal{R}_E(\delta_1) \longrightarrow D(V_x^p) \longrightarrow \mathcal{R}_E(\delta_2) \longrightarrow 0,$$ 
where  $ \delta_1\delta_2^{-1} \neq x^{\mathbb{Z}}, \chi x^{\mathbb{N}}.$
Also note that, by proposition 12, under these assumptions the deformation functor $\mathcal{D}_{V_x^p}^{h, \varphi}$ is pro-representable by a smooth complete local noetherian ring of dimension 3.  Furthermore, by corollary 1, $\mathcal{D}_{V_x^p}^{\Box, h}$ is pro-represented by a smooth complete local noetherian ring of dimension 6, denoted $R_{V_x^p}^{\Box, h}$.

\subsection*{The Framed $X_{fs}$ Space}
Assume that $S$ now contains the prime $l$, where $l$ and $p$ are distinct.  Fix an embedding $\bar{\q}\subset \bar{\q}_l$.   Let $\Sigma= \{l, p\}$ and assume that $S \setminus \{l,p, \infty\} \neq \emptyset$. As in \S4.1, we define the deformation functor $\mathcal{D}_{\q, S}^{\Box}$ which to every $A \in \arwf$ assigns the set of isomorphism classes of tuples $(V_A, \{\beta_{\vi, A}\}_{\vi \in \Sigma})$, where $V_A$ is a deformation of $\vf$ as a $G_{\mathbb{Q}, S}$ representation and $\beta_{\vi, A}$ is an $A$-basis.  This functor is always pro-representable and we denote the universal deformation space by $\xfb$.
\\ \\
\noindent
There is a natural surjective morphism of rigid spaces $\xfb\rightarrow \xff$, induced by the forgetful functor.  This map is formally smooth of relative dimension 7.  We naturally get a morphism $\xf\times \mathbb{G}_m \rightarrow \xff\times \mathbb{G}_m$, acting as the identity on $\mathbb{G}_m$.  This allows us to form the fibre product:

$$\xymatrix{ X_{fs}^{\Box} \ar[d]\ar[r] &\xf\times \mathbb{G}_m \ar[d] \\
X_{fs}\ar[r]&\xff\times \mathbb{G}_m}$$

The natural morphism $X_{fs}^{\Box} \rightarrow X_{fs}$ is surjective and formally smooth of relative dimension 7.  
\\ \\
\noindent
Let $V_{\mathbb{F}}^l$ denote the restriction of $\vf$ to the decomposition group at $l$.  Let $\xflb$ denote the associated universal framed (one framing) deformation space.
Note there is a natural functorial morphism $\xf \rightarrow \xflb$.  This induces a canonical morphism $X_{fs}^{\Box}\rightarrow \xflb$.  Let $(x, f)\in X_{fs}^{\Box}$ with residue field $E$.  Let $x' \in \xflb$ be the image of $(x, \lambda)$ under the above morphism.  Let $R_x'$ denote the complete local ring at $x'$.  By construction there is a natural embedding $k(x')\subset E$, where $k(x')$ is  residue field at $x'$. Let us write $V_x^l$ for the restriction of  $V_x$ to $\gql$.   We write $R_{V_{x}^l}^{\Box}: = R_{x'}\hat{\otimes}_{k(x')}E \in \are$.
By \S2 this ring pro-represents the usual framed deformation functor associated to $V_x^l$.
\\ \\
\noindent
We say that $(x, f)\in X_{fs}^{\Box}$ satisfies $(\ddagger)$ if its image in $X_{fs}$ does.
\begin{prop}
Let $(x, f)\in X_{fs}^{\Box}$ have residue field $E$.  Assume that $(x, \lambda)$ satisfies $(\ddagger)$, then 

$$\hat{\mathcal{O}}_{X_{fs}^{\Box}, (x, \lambda)} \cong (R_{V_x^p}^{\Box, h} \hat{\otimes}_E R_{V_{x}^l}^{\Box})[[x_1, \cdots, x_r]] / (f_1, \cdots, f_{r+2}).$$

\end{prop}

\begin{proof}
Let $\mathcal{D}_{V_x}^{\Box}$ denote the functor which assigns to every $A\in ob(\are)$  the set of isomorphism classes of tuples $(V_A, \{\beta_{\vi, A}\}_{\vi \in \Sigma})$, where $V_A$ is a deformation of $V_x$ as a $G_{\mathbb{Q}, S}$ representation and $\beta_{\vi, A}$ is an $A$-basis.  

Let $\mathcal{D}_{V_x^p}^{\Box}$ the the usual framed deformation functor of $V_x^p$. By usual, we mean with only one framing. This functor is pro-represented by a complete local noetherian ring , which, as usual, we denote by $R_{V_x^p}^{\Box}$.   Observe that there is a natural transformation: $\mathcal{D}_{V_x}^{\Box} \rightarrow \mathcal{D}_{V_x^p}^{\Box}$.  There is also a forgetful natural transformation: $\mathcal{D}_{V_x^p}^{\Box, h}\rightarrow \mathcal{D}_{V_x^p}^{\Box}$.  Exactly the same proof as in the unframed case shows that because of $(\ddagger)$ the complete local ring at $(x, \lambda)$ pro-represents the functor:

$$\mathcal{D}_{V_x}^{\Box} \times_{\mathcal{D}_{V_x^p}^{\Box}}\mathcal{D}_{V_x^p}^{\Box,h}.$$
If the $\mathcal{D}_{V_x^l}^{\Box}$  is the usual framed deformation functor of $V_x^l$, then we denote its universal noetherian complete local ring by $R_{V_x^l}^{\Box}$.  Observe that by assumption $S\setminus \{l,p, \infty\} \neq \emptyset$.  Hence, by proposition 14, we know that $\mathcal{D}_{V_x}^{\Box}$ is pro-represented by a complete local ring of the form

$$R_{V_x}^{\Box} = (R_{V_x^p}^{\Box} \hat{\otimes}_E R_{V_{x}^l}^{\Box})[[x_1, \cdots, x_r]] / (f_1, \cdots, f_{r+s}),$$
where $s =  dim_E(H^0(\gv, adV_x))$ and $v = \infty$. Because $V_x$ is odd this implies that $s =2$.  We deduce that $$\mathcal{D}_{V_x}^{\Box} \times_{\mathcal{D}_{V_x^p}^{\Box}}\mathcal{D}_{V_x^p}^{\Box,h}$$
is pro-represented by a ring of the from:
$(R_{V_x^p}^{\Box, h} \hat{\otimes}_E R_{V_{x}^l}^{\Box})[[x_1, \cdots, x_r]] / (f_1, \cdots, f_{r+2}).$
 Combining all of the above we deduce that 
 $$\hat{\mathcal{O}}_{X_{fs}^{\Box}, (x, \lambda)} \cong (R_{V_x^p}^{\Box, h} \hat{\otimes}_E R_{V_{x}^l}^{\Box})[[x_1, \cdots, x_r]] / (f_1, \cdots, f_{r+2}).$$

\end{proof}

\subsection*{Level Lowering and Level Raising on $X_{fs}$}
Recall that we have the closed subspaces:
$$\xflbsp\subset \xflb$$
and
$$\xflbn\subset \xflb.$$
\noindent
Recall that to any  $x \in \xflb$ we may associate a 2-dimensional Frobenius semi-simple Weil-Deligne representation $(\rho_x, N_x)$.  As in \S 1, let  $\pi$ denote the Tate normalised local Langlands correspondence.  Recall that $x\in \xflbsp$ if and only if $\pi(\rho_x, N_x)$ is a subquotient of a reducible principal series representation.  
\\ \\
\noindent
Similarly, recall that $x \in \xflbn$ if any only if $N_x=0$.  Note that $\xflbsp$ and $\xflbn$ intersect precisely when $\pi(\rho_x, N_x)$ is one dimensional.  
\\ \\
\noindent 
We form the following fibre products:
$$\xymatrix{ X_{fs}^{\Box, sp} \ar[d]\ar[r] &\xflbsp \times \mathbb{G}_m \ar[d] \\
X_{fs}^{\Box}\ar[r]&\xff\times \mathbb{G}_m}$$
and
$$\xymatrix{ X_{fs}^{\Box, N=0} \ar[d]\ar[r] &\xflbn \times \mathbb{G}_m \ar[d] \\
X_{fs}^{\Box}\ar[r]&\xff\times \mathbb{G}_m}$$
Let $X_{fs}^{sp}$ denote the image of the composition $\xflbsp\rightarrow X_{fs}^{\Box}\rightarrow X_{fs}$. Similarly let $X_{fs}^{N=0}$ denote the image of the composition  $\xflbn\rightarrow X_{fs}^{\Box}\rightarrow X_{fs}$
\\ \\
\noindent
To summarise we have the following diagram:

$$\xymatrix{ X_{fs}^{\Box, sp} \ar[d]\ar[r] &X_{fs}^{\Box}\ar[d]  &X_{fs}^{\Box, N=0} \ar[d]\ar[l]\\
X_{fs}^{sp }\ar[r]&X_{fs} &X_{fs}^{N=0}\ar[l] }$$
All vertical arrows are surjective  and at points satisfying $(\ddagger)$ are of relative dimension $7$. All horizontal arrows are closed embeddings.
\\ \\
\noindent
Let $(x, \lambda) \in X_{fs}^{\Box, sp}$, with residue field $E$.  Let $x' \in \xflbsp$ denote the image of $(x, \lambda)$ under the natural projection.  Let $R_{x'}^{sp}$ be the complete local ring at $x'$.  As before there is a natural embedding $k(x')\subset E$, where $k(x')$ is the residue field at $x'$.  We form the complete tensor product $R_{V_x^l}^{\Box, sp}: = R_{x'}^{sp}\hat{\otimes}_{k(x')}E \in \are$.

\begin{prop}
If $(x, \lambda) \in X_{fs}^{\Box, sp}$ satisfies $(\ddagger)$, then there is an isomorphism (in $\are$)

$$\hat{\mathcal{O}}_{X_{fs}^{\Box, sp}, (x, \lambda)} \cong (R_{V_x^p}^{\Box, h} \hat{\otimes}_E R_{V_{x}^l}^{\Box, sp})[[x_1, \cdots, x_r]] / (f_1, \cdots, f_{r+2}). $$
\end{prop}
\begin{proof}By proposition 15 there is an  isomorphism (in $\are$)

$$\hat{\mathcal{O}}_{X_{fs}^{\Box}, (x, \lambda)} \cong (R_{V_x^p}^{\Box, h} \hat{\otimes}_E R_{V_{x}^l}^{\Box})[[x_1, \cdots, x_r]] / (f_1, \cdots, f_{r+2}).$$

Recall that $R_{V_{x}^l}^{\Box}$  is the completed tensor product of the complete local ring at the image of $x$ in $\xflb$ with $E$.  By the definition of $X_{fs}^{\Box, sp}$ we deduce that the complete local ring at $(x, \lambda) \in X_{fs}^{\Box, sp}$ must be equal to the completed tensor product of $\hat{\mathcal{O}}_{X_{fs}^{\Box}, (x, \lambda)}$ with $R_{V_x^l}^{\Box, sp}$ over $R_{V_x^l}^{\Box}$, hence we deduce that 
$$\hat{\mathcal{O}}_{X_{fs}^{\Box, sp}, (x, \lambda)} \cong (R_{V_x^p}^{\Box, h} \hat{\otimes}_E R_{V_{x}^l}^{\Box, sp})[[x_1, \cdots, x_r]] / (f_1, \cdots, f_{r+2}). $$
\end{proof}
\noindent
Let $(x, \lambda) \in X_{fs}^{\Box, N=0}$, with residue field $E$.  Let $x' \in \xflbn$ denote the image of $(x, \lambda)$ under the natural projection.  Let $R_{x'}^{N=0}$ be the complete local ring at $x'$.  As before there is a natural embedding $k(x')\subset E$, where $k(x')$ is the residue field at $x'$.  We form the complete tensor product $R_{V_x^l}^{\Box, N=0}: = R_{x'}^{N=0}\hat{\otimes}_{k(x')}E \in \are$.
\begin{prop}
If $(x, \lambda) \in X_{fs}^{\Box, N=0}$ satisfies $(\ddagger)$, then there is an isomorphism (in $\are$)

$$\hat{\mathcal{O}}_{X_{fs}^{\Box, N=0}, (x, \lambda)} \cong (R_{V_x^p}^{\Box, h} \hat{\otimes}_E R_{V_{x}^l}^{\Box, N=0})[[x_1, \cdots, x_r]] / (f_1, \cdots, f_{r+2}). $$
\end{prop}
\begin{proof}Exactly the same as above.
\end{proof}

We are now in a position to state the two main theorems of this section:

\begin{thm}

If $(x, \lambda) \in X_{fs}^{sp} $ satisfies $(\ddagger)$, then it lies on a irreducible component of dimension greater than or equal to one.
\end{thm}
\begin{proof} We know that at points which satisfy $(\ddagger)$ the morphism $X_{fs}^{\Box, sp} \rightarrow X_{fs}^{sp}$ is formally smooth of relative dimension 7.  Hence if we choose a point lying over $(x, \lambda)$ and determine the dimension of its complete local ring we can deduce the dimension of $\hat{\mathcal{O}}_{X_{fs}^{sp}, (x, \lambda)}$.  By an abuse of notation let $(x, \lambda) \in X_{fs}^{\Box, sp}$ be such a point.  We know by the proposition 16 that 
$$\hat{\mathcal{O}}_{X_{fs}^{\Box, sp}, (x, \lambda)} \cong (R_{V_x^p}^{\Box, h} \hat{\otimes}_E R_{V_{x}^l}^{\Box, sp})[[x_1, \cdots, x_r]] / (f_1, \cdots, f_{r+2}). $$
By  $(\ddagger)$ we know, by corollary 1, that $R_{V_x^p}^{\Box, h}$ is formally smooth of dimension $6$.   By theorem 3 we know that $\xflb$ is formally smooth of dimension 4.  Hence we deduce that $dim(\hat{\mathcal{O}}_{X_{fs}^{\Box, sp}, (x, \lambda)} )\geq 8$.  Thus $dim(\hat{\mathcal{O}}_{X_{fs}^{sp}, (x, \lambda)} )\geq 8 - 7 =1$.  The result follows.
\end{proof}
\begin{thm}
If $(x, \lambda)\in X_{fs}^{N=0} $ satisfies $(\ddagger)$, then it lies on an irreducible  component of dimension greater than or equal to one.
\end{thm}
\begin{proof} The proof is identical to the above theorem.
\end{proof}

\section{The Coleman-Mazure Eigencurve}
Let $S$ be a finite set of places of $\mathbb{Q}$ containing $p$ and $\infty$.    Let $\vf$ be a two dimensional $\f$-vector space equipped with a continuous, odd action of $G_{\mathbb{Q}, S}$.  Let us furthermore assume that $end_{\f[G_{\mathbb{Q}, S}]}\vf = \f$ and that $\vf$ is modular.  By modular we mean that up to semi-simplification it is isomorphic to a mod $p$ representation attached to a classical cuspform. As usual we let $\x$ denote the universal deformation space associated to $\vf$.  

Coleman and Mazur show that the points $(x, \lambda) \in \x \times \mathbb{G}_m$ such that $V_x$ corresponds to a cuspidal eigenform $f$ of level $\Gamma(Mp^r)$ with $r \geq 1$ and $U_pf= \lambda f$, interpolate to a rigid analytic curve $\mathcal{E} \subset \x \times \mathbb{G}_m$.  Here $M$ is a natural number whose divisors are the finite primes in $S$ not equal to $p$.  
The possible $M$ are bounded by the ramification properties of $\vf$ (see \cite{BOC}).  Let $N$ the maximal such $M$.   $\mathcal{E}$ is the cuspidal Coleman Mazur eigencurve of tame level $N$ lying over $\vf$.  

The $\cp$-valued points of $\mathcal{E}$ correspond bijectively to systems of Hecke eigenvalues associated to cuspidal, overconvergent, eigenforms of tame level $N$, which are of finite slope (i.e. the eigenvalue of $U_p$ is non-zero) and whose residual Galois representations have the same semi-simplification as $\vf$.  It is systems of Hecke eigenvalues and not eigenforms themselves because if $N\neq 1$ then if $l$ is a prime dividing $N$ then given a cuspform which is new at $l$ then we can $l$-stabilse to get two old forms at $l$ which have the same associated Galois representation and $U_p$ eigenvalue, hence correspond to the same point on $\mathcal{E}$.  The Hecke theoretic construction of $\mathcal{E}$ (see \cite{KB}) involves \textit{glueing} suitable Hecke algebras away from the tame level $N$.  If we carry out the same construction with full Hecke-algebras (i.e. including for those primes dividing $N$) then we get a curve whose $\cp$-valued points correspond bijectively to cuspidal, overconvergent, eigenforms of tame level $N$, which are of finite slope.  Let $\tilde{\mathcal{E}}$ denote the component of this curve whose points have associated residual representation isomorphic to the semi-simplification of $\vf$.  There is a natural morphism of rigid analytic curves $\tilde{\mathcal{E}} \rightarrow \mathcal{E}$. 

 Let $\mathcal{E}^o\subset \mathcal{E}$ denote the locus whose associated global Galois representation is absolutely irreducible.  This is a dense,  admissible open subset containing all points corresponding to classical cuspforms.  Similarly let $\tilde{\mathcal{E}}^o$ denote the irreducible locus of $\tilde{\mathcal{E}}$.  To any point $x \in \tilde{\mathcal{E}}^o$ we can associate (see \cite{PAU}) a smooth, admissible representation of $GL_2(\ql)$, $\pi_{x,l}$,  and a 2-dimensional Weil-Deligne representation $(\rho_{x, l}, N_x)$.  On any irreducible component of $\tilde{\mathcal{E}}^o$  these representations are generically principal series, special or supercuspidal. As in \S ,  let $\pi$ denote the Tate normalised local Langlands correspondence.  The central result of \cite{PAU} is :
 
 \begin{thmlg}
There is a discrete subset $\mathcal{X}\subset \tilde{\mathcal{E}}^o$ such that for all $x \in  \tilde{\mathcal{E}}^o$, $x\notin \mathcal{X}$, local to global compatibility holds, i.e.
 $$ \pi_{x,l} \cong \pi(\rho_{x, l}, N_x).$$
 Moreover, local to global compatibility can fail under one of two situations:

\begin{enumerate}
\item $x$ lies on a generically principal series irreducible component but $\pi(\rho_{x, l}, N_x)$ is one dimensional.
\item $x$ lies on a generically special irreducible component but $N_x=0$.
\end{enumerate}
 
 \end{thmlg}
Observe that the failure of local to global compatibility is a consequence of the properties of the Galois representation associated to $x$.  It is, therefore, natural to ask what consequences the failure of local to global compatibility has for $\mathcal{E}$.  In particular we have the level raising and lowering conjectures in \S1.

By work of Kisin (\cite{Kis4}) we know that $\mathcal{E} \subset X_{fs}$.  Furthermore in \S11 of \cite{Kis4} it is conjectured that they are equal.  In \cite{Kis6} Kisin proves that $\mathcal{E}=X_{fs}^o\subset X_{fs}$, where $X_{fs}^o$ is the union of the irredicible components containing a potentially semistable point.  In an important forthcoming paper (\cite{EM}) Emerton proves, by establishing a local to global principle in the $p$-adic Langlands programme for $GL_{2/\q}$, the following important result:

\begin{thme}
Let $V_{\f}^p$ be the restriction of $\vf$ to $\gqp$.  If $V_{\f}^p$ is $p$-distinguished and not a twist of an extension of the cyclotomic character by the trivial character then

$$\mathcal{E} = X_{fs}.$$

\end{thme}
Philosophically this is an $R=T$ result.  $X_{fs}$ is constructed Galois theoretically and $\mathcal{E}$ Hecke theoretically.    We are now in a position to be prove (under mild technical restrictions) the level lowering and level raising conjectures. 

\begin{thmll} Let $S$ be a finite set of places of $\q$ such that $\{l,p, \infty\}\subset S$ and $S \setminus \{l,p, \infty\} \neq \emptyset $.  Let $\vf$ be a 2-dimensional, mod $p$,  modular representation of $\gqs$ such that $end_{\f[G_{\mathbb{Q}, S}]}\vf = \f$.  Furthermore, assume that $V_{\f}^p$ is $p$-distinguished and not a twist of an extension of the cyclotomic by trivial characters.   Let $\mathcal{E}$ denote the cuspidal eigencurve lying over $\vf$.  Let $x \in \mathcal{E}$ such that 
\begin{enumerate}
\item The Galois representation associated to $x$ satisfies $(\ddagger)$.
\item $x$ lies on a generically special (at $l$) irreducible component but has trivial monodromy (i.e. $N_x =0$).
\end{enumerate}
Then $f$ lies on a one dimensional generically principal series (at $l$) irreducible component.

\end{thmll}

\begin{proof} The assumption on $V_{\f}^p$ ensures that $\mathcal{E}= X_{fs}$.  By the theorem 6 we observe that there is a one dimension component passing through $x$ on which monodromy is identically zero.   Recall that classical points on such a component must be principal series, thus because they are Zariski dense the whole component is generically principal series.
\end{proof}
\begin{thmlr} Let $S$ be a finite set of places of $\q$ such that $\{l,p, \infty\}\subset S$ and $S \setminus \{l,p, \infty\} \neq \emptyset $.  Let $\vf$ be a 2-dimensional, mod $p$,  modular representation of $\gqs$ such that $end_{\f[G_{\mathbb{Q}, S}]}\vf = \f$.  Furthermore, assume that $V_{\f}^p$ is $p$-distinguished and not a twist of an extension of the cyclotomic by trivial characters.   Let $\mathcal{E}$ denote the cuspidal eigencurve lying over $\vf$.  Let $f \in \mathcal{E}$ such that 
\begin{enumerate}
\item The Galois representation associated to $x$ satisfies $(\ddagger)$.
\item $x$ lies on a one dimensional generically principal series (at $l$)  irreducible component but $\pi(\rho_{x,l}, N_x)$ is one dimensional.
\end{enumerate}
Then $x$ lies on a one dimensional generically special (at $l$) irreducible component.

\end{thmlr}
\begin{proof}
The assumption on $V_{\f}^p$ ensures that $\mathcal{E}= X_{fs}$.  The by theorem 5 we know that $x$ must lie on one dimensional irreducible component which is generically special.
\end{proof}

These two theorems may  be amalgamated into the following central result:

\begin{thm} 
Let $S$ be a finite set of places of $\q$ such that $\{l,p, \infty\}\subset S$ and $S \setminus \{l,p, \infty\} \neq \emptyset $.  Let $\vf$ be a 2-dimensional, mod $p$,  modular representation of $\gqs$ such that $end_{\f[G_{\mathbb{Q}, S}]}\vf = \f$.  Furthermore' assume that $V_{\f}^p$ is $p$-distinguished and not a twist of an extension of the cyclotomic by trivial characters.   Let $\mathcal{E}$ denote the  cuspidal eigencurve lying over $\vf$.  Let $x \in \mathcal{E}$ with associated global Galois representation $V_x$.  We assume that 
\begin{enumerate}
\item $V_x$  and $V_x^p$ are absolutely irreducible.
\item The generalised Hodge-Tate weights of $V_x^p$ are $0$ and $k \notin -\mathbb{N}\cup \{0\}$
\item $\pi(\rho_{x,l}, N_{x})$ is one dimensional.
\end{enumerate}
Then there exist one dimensional irreducible components $\mathcal{Z}, \mathcal{Z}' \subset \mathcal{E}$, generically special and principal series respectively, such that $x \in \mathcal{Z}\cap\mathcal{Z}'$.
\end{thm}
\begin{proof}
By our assumptions we know that $V_x$ cannot be classical, hence by \cite{Kis6} cannot be potentially semi-stable.  Hence $V_x$ satisfies $(\ddagger)$.  $x$ must lie on a one dimensional irreducible component which is either generically special or principal series.  The result follows by the level raising and lowering theorems.
\end{proof}

\bibliographystyle{plain}
\bibliography{bibliography}

\end{document}